\documentclass[11pt]{amsart}
\usepackage{amsmath}
\usepackage{tikz}

\theoremstyle{plain}
\newtheorem{theorem}{Theorem}[section]
\newtheorem{lemma}[theorem]{Lemma}

\newtheorem{cor}[theorem]{Corollary}

\theoremstyle{definition}
\newtheorem{remark}[theorem]{Remark}

\numberwithin{equation}{section}

\def\be{\begin{equation}}
\def\ee{\end{equation}}

\begin{document}

\title[Minimal Graphs in the Hyperbolic Space]
{Boundary Regularity of Minimal Graphs\\ in the Hyperbolic Space}
\author[Qing Han]{Qing Han}
\address{Department of Mathematics\\
University of Notre Dame\\
Notre Dame, IN 46556, USA} \email{qhan@nd.edu}
\author[Xumin Jiang]{Xumin Jiang}
\address{Department of Mathematics\\
Fordham University\\
Bronx, NY 10458, USA}  \email{xjiang77@fordham.edu}

\begin{abstract}
F.-H. Lin \cite{Lin1989Invent}  studied minimal graphs of the Dirichlet problem 
in the hyperbolic space and proved that any such minimal graph has the same global regularity
as the boundary if 
the dimension of the minimal graph is even and that there is an obstacle to the higher regularity 
if the dimension is odd. We discuss the odd dimension case and study how the higher regularity is 
obstructed. 
We introduce the logarithm of the distance to the boundary 
as an additional independent self-variable and establish concise boundary regularity for
the minimal graph.
\end{abstract}

\maketitle

\section{Introduction}\label{sec-Intro}

Complete minimal hypersurfaces in the hyperbolic space
$\mathbb H^{n+1}$ have been studied from different perspectives. 
Anderson \cite{Anderson1982Invent}, \cite{Anderson1983} 
proved the existence of area-minimizing integral $n$-currents in $\mathbb H^{n+1}$
which are asymptotic to given closed embedded $(n-1)$-dimensional submanifolds 
at the infinity of $\mathbb H^{n+1}$. 
These currents are  embedded smooth submanifolds
if $n\le 6$ and may have closed singular sets of Hausdorff dimension at most $n-7$ if $n\ge 7$. 
Hardt and 
Lin \cite{Hardt&Lin1987}  discussed the $C^1$-boundary regularity of such currents. 
In a special setting, Lin \cite{Lin1989Invent} studied the higher boundary regularity.

Assume $\Omega$ is a bounded domain in $\mathbb{R}^n$. 
Lin \cite{Lin1989Invent} studied the 
Dirichlet problem 
\begin{align}\label{Eqf}
\Delta f - \frac{f_i f_j}{1+|D f|^2}f_{ij}+\frac{n}{f} &=0  \quad \text{in } \Omega, \\
\label{EqfCondition} f &=0 \quad \text{on } \partial \Omega.
\end{align}
If $\Omega$ is a $C^2$-domain in $\mathbb R^n$ with 
a nonnegative boundary mean curvature  
$H_{\partial\Omega}\ge 0$ with respect to the inward normal of 
$\partial\Omega$, then \eqref{Eqf} and \eqref{EqfCondition}
admit a unique solution $f\in C(\bar\Omega)\cap C^\infty(\Omega)$ with $f>0$ in $\Omega$. 
Moreover,  the graph of $f$ 
is a complete minimal hypersurface in the hyperbolic space $\mathbb H^{n+1}$
with the asymptotic boundary $\partial\Omega$. 
At each point 
of the boundary, the gradient of $f$ blows up and hence the graph of $f$ 
has a vertical tangent plane. Han, Shen and Wang \cite{Han16CalVar} proved the optimal regularity  
$f\in C^{\frac{1}{n+1}}(\bar\Omega)$. 
Concerning the higher regularity of the graph of $f$, 
Lin \cite{Lin1989Invent}  proved that {\it if $\partial\Omega$ is $C^{n,\alpha}$ for some $\alpha\in (0,1)$, 
then the graph of $f$ is $C^{n,\alpha}$ up to the boundary}. 
Moreover,   Lin \cite{Lin1989Invent} and Tonegawa \cite{Tonegawa1996MathZ}
proved that 
{\it if $\partial\Omega$ is smooth, then the graph of $f$ is smooth up to the boundary 
if the dimension $n$ is even or if the dimension $n$ is odd and 
the principal curvatures of $\partial\Omega$ satisfy a differential equation of order $n+1$.}
See also \cite{Lin2012Invent}.

To study the higher regularity of the graph of $f$, Lin \cite{Lin1989Invent} introduced a new equation. 
Locally near each boundary point, the graph of $f$ can be represented by a function 
over its vertical tangent plane. Specifically, 
we fix a boundary point of $\Omega$, say the origin, and assume that 
$e_n=(0,\cdots, 0,1)$ is the interior normal vector to $\partial\Omega$ 
at the origin. With $x=(x',x_n)$, the $x'$-hyperplane is the tangent plane of 
$\partial\Omega$ at the origin, and the boundary $\partial\Omega$ can be expressed 
in a neighborhood of the origin as a graph of a function over $\mathbb R^{n-1}\times\{0\}$, 
say $x_n=\varphi(x').$
We now denote points in $\mathbb R^{n+1}=
\mathbb R^n\times\mathbb R$ by $(x',x_n,y_n)$.
With $y'=x'$, we represent the graph of $f=f(x', x_n)$ by the graph of a new function $u=u(y', y_n)$,
for small $y'$ and $y_n$, with $y_n>0$. 
Then, for some $R>0$,
$u$ satisfies 
\begin{align}\label{eq-Intro-Equ}
\Delta u - \frac{u_i u_j}{1+|D u|^2}u_{ij}-\frac{n u_{n}}{y_n}&=0  \quad \text{in } B_R^+,\\
\label{eq-Intro-EquCondition}
u&=\varphi\quad\text{on }B_R'.
\end{align}
Lin and Tonegawa proved their regularity results by studying \eqref{eq-Intro-Equ}-\eqref{eq-Intro-EquCondition}. 
In particular, if $\varphi$ is smooth, then $u$ is smooth up to $\{y_n=0\}$ if $n$ is even 
and is $C^{n,\alpha}$ up to $\{y_n=0\}$, for any $\alpha\in (0,1)$, if $n$ is odd.

In this paper, we study obstructions to the higher regularity in the odd dimension 
and aim to prove concise regularity results. 
We present two results under the assumption of smoothness and finite differentiability, respectively.

\begin{theorem}\label{thrm-Main-infinite}
Let $n\ge 3$ be an odd integer,  $\varphi\in C^{\infty}(B'_R)$ be a given function, 
and $u\in C(\bar B^+_R)\cap C^\infty(B^+_R)$ be 
a solution of \eqref{eq-Intro-Equ}-\eqref{eq-Intro-EquCondition}. 
Then, $u$ is smooth in $y', y_n$, and $y_n^n\log y_n$ in $\bar{B}_r$,  
for  any $r\in (0,R)$. 
If, in addition, $\partial_{y_n}\partial_{y_n^{n}\log y_n}u|_{y_n=0}=0$ on $B_R'$, 
then $u\in C^{\infty}(\bar{B}^+_r)$,  
for any $r\in (0,R)$. 
\end{theorem} 

Theorem \ref{thrm-Main-infinite} asserts a concise boundary regularity for $u$ if we view
$y_n^n\log y_n$ as an additional self-variable. 
The formulation of Theorem \ref{thrm-Main-infinite} is inspired by our work in \cite{Han&Jiang-analyticity}
on a conjecture by Lin \cite{Lin1989Invent} concerning the analyticity of the minimal surfaces up to boundary. 
If $\varphi$ is analytic in $B_R'$, we proved that solutions of \eqref{eq-Intro-Equ}-\eqref{eq-Intro-EquCondition} 
are analytic in $y', y_n$ up to $\{y_n=0\}$, if $n$ is even, and are analytic in $y', y_n, y^n_n\log y_n$, if $n$ is odd.

\begin{theorem}\label{thrm-Main-finite}
Let $n\ge 3$ be an odd integer, 
$\varphi\in C^{\ell,\alpha}(B'_R)$ be a given function, 
for some integer $\ell\ge n+1$ and some constant $\alpha\in (0,1)$,  
and $u\in C(\bar B^+_R)\cap C^\infty(B^+_R)$ be 
a solution of \eqref{eq-Intro-Equ}-\eqref{eq-Intro-EquCondition}. 
Assume $m<\ell/n\le m+1$ for some $m\ge 1$. 
Then, 
there exist functions  $w_0, w_1, \cdots, w_m\in C^{\ell, \epsilon}(\bar{B}^+_r)$,  
for any $\epsilon\in (0,\alpha)$ and any $r\in (0,R)$, 
such that
\begin{equation}\label{eq-Main-decomposition-finite-Int}
u=w_0+w_1\log y_n+\cdots+w_m(\log y_n)^m\quad\text{in }B_R^+,\end{equation}
and, for each $j=1, \cdots, m$, 
\begin{equation}\label{eq-Main-vanishing-finite-Int}
\partial_{n}^iw_j(\cdot, 0)=0\quad\text{on }B_R'\quad\text{for any }0\le i\le jn.
\end{equation}
If, in addition, $\partial_n^{n+1}w_1(\cdot, 0)=0$ on $B_R'$, then $w_1, \cdots, w_m$ 
are identically zero and $u\in C^{\ell, \epsilon}(\bar{B}^+_r)$,  
for any $\epsilon\in (0,\alpha)$ and any $r\in (0,R)$. 
\end{theorem} 

The main assertion in Theorem \ref{thrm-Main-finite} is the regularity of 
the functions $w_0, w_1, \cdots, w_m$, which are almost as regular as $\varphi$, 
with a slight loss  from $C^{\ell,\alpha}$ to $C^{\ell, \epsilon}$ for any $\epsilon\in (0,\alpha)$. 

Theorem \ref{thrm-Main-infinite} and Theorem \ref{thrm-Main-finite} 
describe precisely how logarithms obstruct the higher regularity 
and  assert that $\log y_n$ 
and its powers provide the {\it only} obstacles. 

We now compare our results with known results for similar problems. 
The singularity of the equation \eqref{eq-Intro-Equ} is mainly responsible for the lack of the higher regularity. 
For equations with similar singularity, regularity of solutions is usually studied in the form of 
asymptotic expansions near boundary. These expansions are 
established for the singular Yamabe problem 
in \cite{ACF1982CMP, Loewner&Nirenberg1974, Mazzeo1991}, 
the complex Monge-Amp\`{e}re equations in \cite{ChengYau1980CPAM, Fefferman1976, LeeMelrose1982}, 
the asymptotically hyperbolic Einstein metrics
in \cite{Anderson2003, Biquad2010, Chrusciel2005, Hellimell2008}, and many other problems. 
All these expansions contain logarithms of the distance function to the boundary and 
are proved mostly in the smooth category. 
We point out that establishing the expansions means to estimate remainders. 

If $\varphi$ is smooth, following Graham and Witten \cite{Graham&Witten1999},  
we can 
write a formal expansion for solutions of \eqref{eq-Intro-Equ}-\eqref{eq-Intro-EquCondition} for the odd $n$
in the following form: for some smooth functions $c_i$ and $c_{i,j}$ in $B_R'$,
\begin{align}\label{b1b-v}u=\varphi+\sum_{i=2,\text{even}}^{n-1}c_iy_n^i
+\sum_{i=n+1}^\infty
\sum_{j=0}^{[(i-1)/n]}c_{i,j} y_n^i (\log y_n)^j\quad\text{in }B^+_R.
\end{align}
As a consequence of Theorem \ref{thrm-Main-infinite}, 
by considering the Taylor expansion of $u$ in terms of $y_n$ and $y_n^n\log y_n$
with coefficients given by smooth functions of $y'$, we obtain remainder estimates 
associated with the infinite series \eqref{b1b-v}. 
Moreover, in Theorem \ref{thrm-Main-finite}, by considering 
Taylor expansions of $w_0, \cdots, w_m$ up to order $\ell$, we obtain 
\begin{equation}\label{eq-expansion-finite}
u=\varphi+\sum_{i=2,\text{even}}^{n-1}c_iy_n^i
+\sum_{i=n+1}^\ell\sum_{j=0}^{[(i-1)/n]}c_{i,j}y_n^i(\log y_n)^j+R_\ell\quad\text{in }B^+_R,\end{equation}
for some functions $c_i\in C^{\ell-i,\alpha}(B_R')$ for $2\le i\le n-1$, 
$c_{i,j}\in C^{\ell-i,\epsilon}(B_R')$, and $R_\ell\in C^{\epsilon}(\bar{B}^+_r)$, 
for any $\epsilon\in (0,\alpha)$ and $r\in (0,R)$.  Under the assumption of the finite differentiability, 
the expansion \eqref{eq-expansion-finite} is optimal in the order of expansion 
and almost optimal in the regularity of coefficients and the remainder, 
with a slight loss only in the H\"older index.

In \eqref{b1b-v} or \eqref{eq-expansion-finite}, the coefficients  
$c_2, \cdots, c_{n-1}, c_{n+1,1}$ 
have explicit expressions in terms of $\varphi$  and are referred to as
{\it local terms}. 
For example, we have, 
for $n=3$, 
\be\label{eq-Intro-c_31}
c_{4,1}=-\frac{1}{8}\sqrt{1+|D_{y'}\varphi|^2}\big\{\Delta_{\Sigma} H+2H(H^2-K)\big\},\ee
where $H$ and $K$ are the mean curvature and the Gauss curvature of
the graph $\Sigma$ given by $x_n=\varphi(y')$, 
respectively. 
The coefficient 
$c_{n+1,0}$ 
is the first nonlocal coefficient and exhibits a slight loss of regularity in the H\"older index. 


The most part of the paper is devoted to the study of a class of quasilinear elliptic equations more general than 
\eqref{eq-Intro-Equ}. Theorem \ref{thrm-Decomposition-finite} and Theorem \ref{thrm-Decomposition-infinite}
 should be considered as the main results
in this paper and can be applied to solutions of \eqref{eq-Intro-Equ}. 
Theorem \ref{Thm-MainThm} plays a fundamental role in our study. 
Results in this paper are established based on PDE techniques, 
such as barrier functions and 
scalings. 
With this approach, 
we are able to track the regularity of coefficients and the remainder of the expansion
and present the estimate of the remainder under the assumption of the optimal regularity, 
as shown in Theorem \ref{Thm-MainThm}.

We finish the introduction with a brief outline of the paper. 
In Section \ref{sec-Setup}, we rewrite the equation \eqref{eq-Intro-Equ} based on a fundamental estimate 
proved by Lin \cite{Lin1989Invent}. 
In Section \ref{Sec-TangentialSmooth}, 
we discuss a class of 
quasilinear elliptic equations with singularity and prove the tangential smoothness near boundary. 
In Section \ref{sec-NormalRegularity}, we treat quasilinear elliptic equations as 
ordinary differential equations and prove the regularity along the normal direction. 
In Section \ref{Sec-Decomposition}, we prove two decomposition results for solutions near boundary. 

We would like to thank Robin Graham, Fang-Hua Lin, and Rafe Mazzeo 
 for many helpful comments and suggestions. 

\section{The Setup}\label{sec-Setup}

In this section, we will rewrite the equation \eqref{eq-Intro-Equ} based on a fundamental estimate 
proved by Lin \cite{Lin1989Invent}.

We denote by $y=(y', t)$ points in $\mathbb R^n$, 
with $y_n=t$, and set, for any $r>0$, 
$$G_r=\{(y', t):\, |y'|<r,\, 0<t<r\}.$$
For some $\varphi\in C^2(B_1')$, we consider a solution 
$u\in C(\bar G_1)\cap C^\infty(G_1)$ of \eqref{eq-Intro-Equ}-\eqref{eq-Intro-EquCondition}, i.e.,  
\begin{align*}
\Delta u - \frac{u_i u_j}{1+|D u|^2}u_{ij}-\frac{n u_{t}}{t}&=0  \quad \text{in } G_R,\\
u&=\varphi\quad\text{on }B_R'.
\end{align*}

\begin{lemma}\label{Lemma-TangentialGradientDecay_u} 
Assume $\varphi\in C^2(B_1')$ and let $u\in C(\bar G_R)\cap C^\infty(G_R)$ be 
a solution of  \eqref{eq-Intro-Equ}
and \eqref{eq-Intro-EquCondition}. 
Then, for any $(y',t)\in G_{R/2}$, 
\begin{align} 
\label{GE2a0}
|(u-\varphi)(y',t)|&\le C_0t^2,\\
\label{GE2a1}
|D(u-\varphi)(y',t) |&\le C_0t,\\
\label{GE2a2}
|D^2(u-\varphi)(y',t) |&\le C_0,
\end{align}
where $C_0$ is a positive constant depending only on $n$, $|u|_{L^\infty(G_R)}$,
and $|\varphi|_{C^{2}(\bar G_R)}$. 
\end{lemma} 

Lemma \ref{Lemma-TangentialGradientDecay_u} is a consequence of Theorem 3.4 and Remark 3.8 in \cite{Lin1989Invent}. 

In view of Lemma \ref{Lemma-TangentialGradientDecay_u}, it is more convenient to study $u-\varphi$. 
Set 
\be\label{eq-NewFunction-g1}v=u-\varphi.\ee
Then, 
\begin{align}\label{eq-Equ-v}
\Delta v - \frac{u_i u_j}{1+|D u|^2}v_{ij}-\frac{n v_{t}}{t}
+\Delta \varphi- \frac{u_i u_j}{1+|D u|^2}\varphi_{ij}=0  \quad \text{in } G_R.
\end{align}
We rewrite \eqref{eq-Equ-v} as 
\begin{equation}\label{eq-Equ-v-general}A_{ij} v_{ij}+P \frac{v_t}{t}+Q \frac{v}{t^2}+N=0\quad\text{in }G_R,
\end{equation} 
where 
\begin{align*}
A_{ij}=\delta_{ij}-\frac{(v+\varphi)_i(v+\varphi)_j}{1+|D(v+\varphi)|^2},\quad
P=-n, \quad Q=0,\end{align*} 
and 
$$N=\Delta \varphi- \frac{(v+\varphi)_i(v+\varphi)_j}{1+|D(v+\varphi)|^2}\varphi_{ij}.$$
It is easy to check 
$$2A_{nn}+2P+Q\le -2(n-1),$$
and, for any $\alpha\in (0,1)$, 
$$(2+\alpha)(1+\alpha)A_{nn}+(2+\alpha)P+Q\le (1+\alpha)(1+\alpha-n).$$
By \eqref{GE2a0}-\eqref{GE2a2}, we have 
\be\label{eq-GE2a-general}
|v|\le C_0t^2, \ |Dv|\le C_0t,\ |D^2v|\le C_0\quad\text{in }G_{R/2}.\ee
In Section \ref{Sec-TangentialSmooth}, we will prove the tangential regularity 
by using the equation \eqref{eq-Equ-v-general}. 

Based on \eqref{eq-Equ-v}, 
a straightforward calculation shows that $v$ satisfies 
\begin{align}\label{eq-ODE_v}
v_{tt}+\frac{p}{t}v_t+\frac{q}{t^2}v+F=0,
\end{align}
where 
$$p=-n,\quad q=0,$$ 
and 
\begin{align}\label{eq-ODE_vF}\begin{split}
F=F(v)&=\Big(\delta_{\alpha\beta}-\frac{u_\alpha u_\beta}{1+|D_{y'}u|^2}\Big)
(v_{\alpha\beta}+\varphi_{\alpha\beta})
-\frac{2u_\alpha}{1+|D_{y'}u|^2}v_tv_{\alpha t} \\
&\qquad+\frac{v_t^2}{1+|D_{y'}u|^2}(\Delta_{y'}v+\Delta_{y'}\varphi)
-\frac{nv_t}{1+|D_{y'} u|^2}\frac{v_t^2}{t}.
\end{split}\end{align}
Then, $F$ depends on $y'$ through derivatives of $\varphi$ up to order 2 
and $F$ is a smooth function in $t$,  
$v_t, D_{y'}v_t, 
D^2_{y'}v,$ and ${v_t^2}/{t}$.
Moreover, $F$ is linear in ${v_t^2}/{t}.$
In Section \ref{sec-NormalRegularity}, we will study the regularity along the normal direction
by using the equation \eqref{eq-ODE_v}. 
Set
$$\underline{m}=0, \quad \overline{m}=n+1.$$
Then, $p=1-(\underline{m}+\overline{m})$ and $q=\underline{m}\cdot\overline{m}$. 
These simple algebraic relations play fundamental roles in Section \ref{sec-NormalRegularity}. 

As we see, in converting \eqref{eq-Intro-Equ} to \eqref{eq-Equ-v-general} and \eqref{eq-ODE_v}, 
all assumptions in Sections \ref{Sec-TangentialSmooth}-\ref{Sec-Decomposition} hold. 
Therefore, we obtain Theorem \ref{thrm-Main-finite} by Theorem \ref{thrm-Decomposition-finite} and 
obtain Theorem \ref{thrm-Main-infinite} by Theorem \ref{thrm-Decomposition-infinite}.

\section{Tangential Smoothness}\label{Sec-TangentialSmooth}

In the present and the next sections, we study a class of 
quasilinear elliptic equations with singularity and discuss the regularity  
of solutions near boundary. We study the regularity along tangential directions in this section and 
along the normal direction in the next section. 
We follow Lin \cite{Lin1989Invent} closely for the proof of the tangential regularity, by 
the maximum principle and rescaling. 

We denote by $y=(y', t)$ points in $\mathbb R^n$, with $y_n=t$, and, set, for any constant $r>0$, 
$$G_r=\{(y',t):\, |y'|<r,\, 0<t<r\}.$$
Throughout this section, we denote by $Du$ and $D_{y'}u$ 
derivatives with respect to $(y', t)$ and $y'$, respectively. 

For a fixed $R>0$,  let $A_{ij}, P, Q$, and $N$ be given functions of $(y',t,p,s)$ with $(y',t)\in \bar{G}_R$, 
$p\in \mathbb R^{n}$, and $s\in \mathbb R$. 
We always assume that $A_{ij}, P, Q$, and $N$ and their derivatives with respect to $y'$ up to a certain order 
are smooth in $t,p,s$ and all the corresponding derivatives are bounded in 
$\bar{G}_R\times\mathbb R^{n+1}.$

Suppose $v\in C(\bar G_R)\cap C^\infty(G_R)$ satisfies  
\begin{align}\label{eq-QuasiMain}
A_{ij} v_{ij}+P \frac{v_t}{t}+Q \frac{v}{t^2}+N=0\quad\text{in }G_R, 
\end{align}
where $A_{ij}=A_{ij}(y',t,Dv, {v}/{t})$ and $P, Q$, and $N$ have similar forms. 
We assume \eqref{eq-QuasiMain} is uniformly elliptic; namely, there exists a positive constant 
$\lambda$ such that, for 
any $\xi\in\mathbb R^n$, 
\begin{equation}\label{eq-ellipticity} 
\lambda^{-1}|\xi|^2\le A_{ij}\xi_i\xi_j\le \lambda|\xi|^2\quad\text{in }{G}_R\times\mathbb R^{n+1}.\end{equation}
Concerning the solution $v$, we assume, for some positive constant $C_0$,  
\be\label{eq-QuasiCondition}
|v|\le C_0t^2,\ |Dv|\le C_0t,\ |D^2v|\le C_0\quad\text{in }G_R.\ee

Now, we derive an estimate of derivatives near boundary by a rescaling method, 
which reduces global estimates to local ones. 

\begin{theorem}\label{thrm-TangentialEstimate}
Assume $A_{ij}, P, Q$, and $N$ are $C^{\ell,\alpha}$ in $y'$, for some 
$\ell\ge 0$ and $\alpha\in (0,1)$, and smooth in other arguments. 
Let $v\in C^{1,1}(\bar G_R)\cap C^{\ell+2, \alpha}(G_R)$ be a solution of \eqref{eq-QuasiMain} in $G_R$, 
for some $R>0$, and satisfy \eqref{eq-QuasiCondition}. 
Assume, for some positive constant $c_0$,  
\be\label{eq-NegativeCondition}2 A_{nn}+2P+Q\le -c_0\quad\text{in }{G}_R\times\mathbb R^{n+1}.\ee
Then, for any $r\in (0,R)$ and any  $\tau=0, 1, \cdots, \ell$, 
\be\label{eq-BoundTangentialGradient_vv}
\frac{D^\tau_{y'}v}{t^2},\, \frac{DD^\tau_{y'}v}{t},\, D^2D^\tau_{y'}v\in L^\infty(G_{r}),\ee
and 
\be\label{eq-BoundTangentialGradient_vv2}
\frac{D^\tau_{y'}v}{t},\, D^\tau_{y'}Dv, \, \frac{D^\tau_{y'}(v^2)}{t^3}, \,
\frac{D^\tau_{y'}(vv_t)}{t^2}, \, 
\frac{D^\tau_{y'}(v_t^2)}{t}\in C^{0,1}(\bar G_{r}).\ee
Moreover, the cooresponding bounds in \eqref{eq-BoundTangentialGradient_vv}
and \eqref{eq-BoundTangentialGradient_vv2}  depend only on $n$, $\ell$, $\alpha$, $r$, 
$\lambda$ in \eqref{eq-ellipticity}, $C_0$ in \eqref{eq-QuasiCondition}, $c_0$ in \eqref{eq-NegativeCondition}, 
and the $C^{\ell,\alpha}$-norms of 
$A_{ij}, P, Q$ and $N$. 
\end{theorem}

\begin{proof}
We first note that \eqref{eq-BoundTangentialGradient_vv} implies 
\eqref{eq-BoundTangentialGradient_vv2}. In fact, the derivatives of 
the expressions in \eqref{eq-BoundTangentialGradient_vv2} can be expressed 
in terms of functions in \eqref{eq-BoundTangentialGradient_vv}. 
For a later reference, we list 
these relations for $\tau=0$ as follows: 
\begin{align*} 
D\left(\frac{v}{t}\right)&=\frac{Dv}{t}-\frac{v}{t^2}e_n,\\
D(Dv)&=D^2v,\end{align*} 
and 
\begin{align*} 
D\Big(\frac{v^2}{t^3}\Big)&=2\frac{v}{t^2}\cdot \frac{Dv}{t}-3\Big(\frac{v}{t^2}\Big)^2e_n,\\
D\Big(\frac{vv_t}{t^2}\Big)&=\frac{Dv}{t}\cdot\frac{v_t}{t}+\frac{v}{t^2}Dv_t
-2\frac{v}{t^2}\cdot\frac{v_t}{t}e_n,\\
D\Big(\frac{v_t^2}{t}\Big)&=2\frac{v_t}{t}Dv_t-\Big(\frac{v_t}{t}\Big)^2e_n.\end{align*}
For $\tau=0$, \eqref{eq-BoundTangentialGradient_vv} follows from 
\eqref{eq-QuasiCondition}. 
We fix an integer $1\le k\le \ell$ and assume \eqref{eq-BoundTangentialGradient_vv}
holds for $\tau=0, \cdots, k-1$. We now consider the case $\tau=k$ and $r=R/2$.

By applying $D_{y'}^k$ to \eqref{eq-QuasiMain}, we obtain 
\begin{align}\label{TS91}
A_{ij} (D_{y'}^kv)_{ij}+P \frac{(D_{y'}^kv)_t}{t}+Q \frac{D_{y'}^kv}{t^2}+N_k=0,
\end{align}
where $N_k$ is given by 
\be\label{eq-KeyExpressions2}N_k=\sum_{\substack{l+m= k\\ m\le k-1}}
a_{lm}\Big(D_{y'}^lA_{ij}\cdot D_{y'}^mv_{ij}+
D_{y'}^lP\cdot\frac{D_{y'}^mv_t}{t}+D_{y'}^lQ\cdot\frac{D_{y'}^mv}{t^2}\Big)
+D_{y'}^kN,\ee
for some constant $a_{lm}$. Derivatives of $A_{ij}, P, Q$ and $N$ also result in derivatives of 
$v$. In conclusion, $N_k$ is a polynomial of the expressions in 
\eqref{eq-BoundTangentialGradient_vv}, for $\tau\le k-1$, except $Dv$ and $v/t$. 
Then, by the induction hypotheses, $N_k$ is bounded in $G_r$. 
In the following, we set 
$$\mathcal Lw=A_{ij} w_{ij}+\frac{P}{t}w_t+\frac{Q}{t^2}w.$$
Consider, for some positive constants $a$ and $b$ to be determined, 
$$\overline{w}(y',t)=a|y'|^2+bt^2.$$
Note $D_{y'}^kv=0$ on $t=0$ and $\pm D_{y'}^kv\le Ct$ by the induction hypothesis. 
Then, 
we can choose
$a$ and $b$ large such that $D_{y'}^kv\le \overline{w}$ on $\partial G_r$. 
Next, 
\begin{align}
\mathcal L\overline w=(2 A_{nn}+2P+Q) b + 2a\sum_{\alpha=1}^{n-1}A_{\alpha\alpha}.\label{TS72}
\end{align}
The assumption \eqref{eq-QuasiCondition} implies that the coefficients 
of $\mathcal L$ are $C^{0,1}$. 
By \eqref{eq-NegativeCondition} and taking $b$ sufficiently large, 
we obtain $\mathcal L\overline{w}\le 
\mathcal L\mathcal (\pm D_{y'}^kv)$ in $G_r$. 
The maximum principle implies $\pm D_{y'}^kv\le \overline{w}$ in $G_r$. By taking $y'=0$, 
we obtain $\pm D_{y'}^kv(0', t)\le bt^2$ for any $t\in (0,r)$. 
For any fixed $y_0'\in B'_{r}$, we  consider, 
instead of $\overline{w}$, 
$$\overline{w}_{y_0'}(y',t)=a|y'-y_0'|^2+bt^2.$$
By repeating the above argument, we conclude 
$$\pm D_{y'}^kv(y',t)\le bt^2\quad\text{for any }(y',t)\in G_r.$$
Therefore, 
\be\label{eq-QuadraticBound}|D_{y'}^kv|\le Ct^2\quad\text{in } G_r.\ee

Next, we prove 
\be\label{eq-QuadraticBound2}|DD_{y'}^kv|\le Ct,\quad |D^2D_{y'}^kv|\le C\quad\text{in }G_{r/2}.\ee
Take any $y_0=(y_0', t_0)\in G_{r/2}$ and set $\delta=t_0/2$. 
By applying the scaled interior $C^{1,\alpha}$-estimates to \eqref{TS91} in $B_{\delta}(y_0)$, we have 
\begin{align*}&\delta|DD_{y'}^kv|_{L^\infty(B_{3\delta/4}(y_0))}
+\delta^{1+\alpha}[DD_{y'}^kv]_{C^{\alpha}(B_{3\delta/4}(y_0))}\\
&\qquad\le C\big(|D_{y'}^kv|_{L^\infty(B_{\delta}(y_0))}
+\delta^2|N_k|_{L^\infty(B_{\delta}(y_0))}\big),\end{align*}
and hence, by \eqref{eq-QuadraticBound},  
$$|DD_{y'}^kv|_{L^\infty(B_{3\delta/4}(y_0))}
+\delta^{\alpha}[DD_{y'}^kv]_{C^{\alpha}(B_{3\delta/4}(y_0))}\le C_1\delta.$$
As a consequence, the scaled $C^\alpha$-norm of  $N_k$ is uniformly bounded in 
$B_{3\delta/4}(y_0)$. By applying the scaled interior $C^{2,\alpha}$-estimates 
to \eqref{TS91} in $B_{3\delta/4}(y_0)$, we have 
\begin{align*}
&\delta^2|D^2D_{y'}^kv|_{L^\infty(B_{\delta/2}(y_0))}
+\delta^{2+\alpha}[D^2D_{y'}^kv]_{C^\alpha(B_{\delta/2}(y_0))}\\
&\quad\le C\big(|D_{y'}^kv|_{L^\infty(B_{3\delta/4}(y_0))}
+\delta^2|N_k|_{L^\infty(B_{3\delta/4}(y_0))}
+\delta^{2+\alpha}[N_k]_{C^{\alpha}(B_{3\delta/4}(y_0))}\big),\end{align*}
and, in particular,  
$$|D^2D_{y'}^kv|_{L^\infty(B_{\delta/2}(y_0))}
+\delta^{\alpha}[D^2D_{y'}^kv]_{C^\alpha(B_{\delta/2}(y_0))}\le C_1.$$
We hence have \eqref{eq-QuadraticBound2}. 
We conclude the proof of \eqref{eq-BoundTangentialGradient_vv} for $\tau=k$ and $r=R/2$. 
\end{proof} 

There is a loss of regularity in Theorem \ref{thrm-TangentialEstimate}. 
Under the assumptions $A_{ij}, P, Q, N$ 
$\in C^{\ell,\alpha}$, we only proved 
$D_{y'}^\ell u\in C^{1, 1}$. We now prove it is $C^{2,\alpha}$ 
under a slightly strengthened condition on coefficients. 

\begin{theorem}\label{thrm-TangentialEstimateAlpha}
Assume $A_{ij}, P, Q$, and $N$ are $C^{\ell,\alpha}$ in $y'$, for some 
$\ell\ge 0$ and $\alpha\in (0,1)$, and smooth in other arguments. 
Let $v\in C^{1,1}(\bar G_R)\cap C^{\ell+2,\alpha}(G_R)$ be a solution of \eqref{eq-QuasiMain} in $G_R$, 
for some $R>0$, and satisfy \eqref{eq-QuasiCondition}. 
Assume \eqref{eq-NegativeCondition} and, for some constant $c_\alpha>0$,  
\be\label{eq-NegativeConditionAlpha}(2+\alpha)(1+\alpha) A_{nn}
+(2+\alpha)P+Q\le -c_\alpha\quad\text{in }G_R.\ee
Then, for any $r\in (0,R)$ and any  $\tau=0, 1, \cdots, \ell$, 
\be\label{eq-BoundTangentialGradient_vvAlpha}
\frac{D^\tau_{y'}v}{t^2},\, \frac{DD^\tau_{y'}v}{t},\, D^2D^\tau_{y'}v\in C^\alpha(\bar G_{r}),\ee
and 
\be\label{eq-BoundTangentialGradient_vv2Alpha}
\frac{D^\tau_{y'}v}{t},\, D^\tau_{y'}Dv, \, \frac{D^\tau_{y'}(v^2)}{t^3}, \,
\frac{D^\tau_{y'}(vv_t)}{t^2}, \, 
\frac{D^\tau_{y'}(v_t^2)}{t}\in C^{1,\alpha}(\bar G_{r}).\ee
Moreover, the cooresponding bounds in \eqref{eq-BoundTangentialGradient_vvAlpha}
and \eqref{eq-BoundTangentialGradient_vv2Alpha}  depend only on $n$, $\ell$, $\alpha$, $r$, 
$\lambda$ in \eqref{eq-ellipticity}, $C_0$ in \eqref{eq-QuasiCondition}, $c_0$ in \eqref{eq-NegativeCondition}, 
$c_\alpha$ in \eqref{eq-NegativeConditionAlpha}, and the $C^{\ell,\alpha}$-norms of 
$A_{ij}, P, Q$, and $N$. 
\end{theorem}

\begin{proof} We fix $r=R/2$. 

{\it Step 1.} 
We first consider $\ell=0$. We claim, for some $c_2\in C^\alpha(B_R')$, some 
$r\in (0,R)$ and any $(y',t)\in G_{r}$, 
\begin{equation}\label{eq-Alpha1}|v(y',t)-c_2(y')t^2|\le Ct^{2+\alpha}.\end{equation}
The expression of $c_2$ will be given in the proof below. We point out that since $c_2$ is only 
$C^\alpha$, we cannot differentiate $c_2$. For convenience, we set 
$$\mathcal L(v)=A_{ij} v_{ij}+P \frac{v_t}{t}+Q \frac{v}{t^2}, \quad 
\mathcal Q(v)=\mathcal L(v)+N.$$
For some function $c_2$ in $B'_R$ and some function $\psi$ in $G_R$ to be determined, we set 
$$\overline{v}=c_2(0)t^2+\psi.$$
A straightforward calculation yields 
$$\mathcal Q(\overline{v})=\mathcal L(\psi)+(2A_{nn}+2P+Q)c_2(0)+N,$$
where $A_{ij}, P, Q$ and $N$ are evaluated at $y, t, D\overline{v}$, and $\overline{v}/t$. In the following, we take 
\begin{equation}\label{eq-Expression-c_2}
c_2(0)=-\Big(\frac{N}{2A_{nn}+2P+Q}\Big)\Big|_0,\end{equation}
where the right-hand side is evaluated with all of its arguments replaced by zero. 
Hence, by the expression of $\overline{v}$ and the $C^\alpha$-regularity of 
$A_{ij}, P, Q$ and $N$, we have 
\begin{equation}\label{eq-Alpha2}\mathcal Q(\overline{v})\le \mathcal L(\psi)+
C\Big(|y'|^\alpha+t^\alpha+|D\psi|^\alpha+\Big(\frac{\psi}{t}\Big)^\alpha\Big).\end{equation}
Next, set, for some constants $\mu_1$ and $\mu_2$ to be determined, 
$$\psi(y',t)=\mu_1t^2(|y'|^2+t^2)^{\frac\alpha2}+\mu_2t^{2+\alpha}.$$
By a straightforward calculation, we get 
$$\mathcal L(\psi)=\mu_1B_1(|y'|^2+t^2)^{\frac\alpha2}+\mu_2B_2t^\alpha,$$
where 
$$\aligned B_1&=2A_{nn}+2P+Q
+\frac{\alpha}{|y'|^2+t^2}(A_{ab}\delta_{ab}t^2+4A_{an}y_at+5A_{nn}t^2+Pt^2)\\
&\qquad+\alpha(\alpha-2)\frac{t^2}{(|y'|^2+t^2)^2}(A_{ab}y_ay_b+2A_{an}y_at+A_{nn}t^2),
\endaligned$$
and 
$$B_2=(2+\alpha)(1+\alpha) A_{nn}
+(2+\alpha)P+Q.$$ 
Recall the assumptions \eqref{eq-NegativeCondition} and 
\eqref{eq-NegativeConditionAlpha}. 
First, we have 
$B_2\le -c_\alpha$ in $G_R.$
Next, we can 
find a constant $M$ such that $B_1\le -c_0/2$, for any $(y',t)\in G_r$ with $|y'|\ge Mt$.
Hence, for such $(y',t)$, we have 
$$\mathcal L(\psi)\le -\frac12c_0\mu_1(|y'|^2+t^2)^{\frac\alpha2}-c_\alpha\mu_2t^{\alpha}
\le -\frac12c_0\mu_1(|y'|^2+t^2)^{\frac\alpha2}.$$
If $|y'|\le Mt$, then $B_1\le C$ and 
$$\aligned 
\mathcal L(\psi)&\le C\mu_1(|y'|^2+t^2)^{\frac\alpha2}-c_\alpha\mu_2t^{\alpha}
\le -\Big(\frac{c_\alpha\mu_2}{(1+M^2)^{\frac\alpha2}}-C\mu_1\Big)(|y'|^2+t^2)^{\frac\alpha2}\\
&=-c\mu_1(|y'|^2+t^2)^{\frac\alpha2},\endaligned$$
by choosing $\mu_2$ to be a constant multiple of $\mu_1$. 
Therefore, we obtain
\begin{equation}\label{eq-Alpha3}
\mathcal L(\psi)\le -c\mu_1(|y'|^2+t^2)^{\frac\alpha2}\quad\text{in }G_r.\end{equation}
We point out that, if $B_1<0$ in $G_r$, we can simply take $\mu_2=0$ and there is no need to assume 
\eqref{eq-NegativeConditionAlpha}.  
By the explicit expression of $\psi$, we have 
\begin{equation}\label{eq-Alpha4}
|D\psi|^\alpha+\Big(\frac{\psi}{t}\Big)^\alpha
\le C\mu_1^\alpha t^\alpha (|y'|^2+t^2)^{\frac{\alpha^2}2}
\le C\mu_1^\alpha  (|y'|^2+t^2)^{\frac{\alpha}2}.\end{equation}
By \eqref{eq-Alpha2}, \eqref{eq-Alpha3}, and \eqref{eq-Alpha4}, we obtain 
$$\mathcal Q(\overline{v})\le -c\mu_1(|y'|^2+t^2)^{\frac\alpha2}+
C\mu_1^\alpha(|y'|^2+t^2)^{\frac\alpha2}.$$
By $\alpha\in (0,1)$, we can take $\mu_1$ sufficiently large such that 
$$\mathcal Q(\overline{v})\le 0\quad\text{in }G_r.$$
We now compare $v$ and $\overline{v}$ on $\partial G_r$. 
By \eqref{eq-QuasiCondition}, in order to have $v\le \overline{v}$ on $\partial G_r$, it 
suffices to require 
\begin{equation}\label{eq-Alpha-Lower}C_0+|c_2(0)|\le \mu_1r^\alpha.\end{equation}
Hence, we can take $\mu_1$ sufficiently large such that 
\eqref{eq-Alpha-Lower} hold. Therefore, we have 
$\mathcal Q(\overline{v})\le \mathcal Q(v)$ in $G_r$ and $v\le \overline{v}$ on $\partial G_r$. 
By the maximum principle, we get $v\le \overline{v}$ in $G_r$ and hence 
$$v\le c_2(0)t^2+\psi\quad\text{in }B_r.$$
Similarly, we have 
$$v\ge c_2(0)t^2-\psi\quad\text{in }B_r.$$ By taking 
$y'=0$, we have \eqref{eq-Alpha1} for $y'=0$. We can prove \eqref{eq-Alpha1} for any 
$(y', t)\in G_r$ by a similar method. Instead of \eqref{eq-Expression-c_2}, we have 
\begin{equation}\label{eq-Expression-c_2General}
c_2(y')=-\Big(\frac{N}{2A_{nn}+2P+Q}\Big)(y', 0).\end{equation}
We note that $c_2\in C^\alpha(B'_R)$. 

With \eqref{eq-Alpha1}, we will prove 
\be\label{eq-BoundTangentialGradient_vvAlpha0}
\frac{v}{t^2},\, \frac{Dv}{t},\, D^2v\in C^\alpha(\bar G_{r}).\ee
This is \eqref{eq-BoundTangentialGradient_vvAlpha} for $\tau=0$. 
To prove \eqref{eq-BoundTangentialGradient_vvAlpha0}, 
we take any $(y_0', t_0)\in G_{r}$ and set $\delta=t_0/2$. 
Instead of $v$, we consider the equation for $v(y',t)-c_2(y'_0)t^2$. 
The rest of the proof is  similar as that in the proof of 
Theorem \ref{thrm-TangentialEstimate}. We omit details and point out that 
\eqref{eq-Alpha1} allows us to scale the estimate of the H\"older semi-norms of the second derivatives. 

{\it Step 2.} We prove for general $\ell$ by an induction. We fix an integer $1\le k\le \ell$ and assume 
\eqref{eq-BoundTangentialGradient_vvAlpha} holds for $\tau=0, \cdots, k-1$. We now consider the case
$\tau=k$. 

We first claim, for some $c_{k,2}\in C^\alpha(B_R')$ and any $(y',t)\in G_{r}$, 
\begin{equation}\label{eq-Alpha6}|D_{y'}^kv(y',t)-c_{k,2}(y')t^2|\le Ct^{2+\alpha}.\end{equation}
The proof is similar as the proof of \eqref{eq-Alpha1}. By the induction hypothesis, the 
coefficients and the nonhomogeneous term in \eqref{TS91} satisfy all the regularity assumptions. 
We omit details. 

With \eqref{eq-Alpha6}, we can prove \eqref{eq-BoundTangentialGradient_vvAlpha}  
for $\tau=k$ by a similar scaling argument. 
\end{proof}

\section{Regularity along the Normal Direction}\label{sec-NormalRegularity}

In this section, we continue our study of the equation \eqref{eq-QuasiMain}
and discuss the regularity along the normal direction. 
In the previous section, we established the regularity of solutions 
along tangential directions. 
As noted by Lin \cite{Lin1989Invent}, the tangential regularity allows us 
to write the underlying partial differential equation as 
an ordinary differential equation in the $t$-direction. 
The utilization of the ODEs in this paper is modified from the work 
by Jian and Wang \cite{JianWang2013JDG, JianWang2014Adv}, 
where they iterated ODEs. 
The main focus there is  the regularity before singularity appears. 
As a result, their iteration of ODEs terminates before 
the logarithmic terms show up.  In our case, analyzing the impact of logarithmic terms 
on certain combinations of derivatives
constitutes an indispensable part of the 
study of the regularity of remainders.  
We will use a single ODE and iterate solutions. 
 
As in Section \ref{Sec-TangentialSmooth},
we denote by $y=(y', t)$ points in $\mathbb R^n$ and set, for any constant $r>0$, 
$$G_r=\{(y',t):\, |y'|<r,\, 0<t<r\}.$$
We start with the equation \eqref{eq-QuasiMain} and rewrite it in the form 
\begin{align}\label{gtt}
v_{tt}+p\frac{v_t}{t}+q \frac{v}{t^2}+F=0, 
\end{align}
where $p$, $q$, and $F$ are given by 
$$p=(A_{nn}^{-1}{P})\big|_{t=0}, \quad q=(A_{nn}^{-1}{Q})\big|_{t=0},$$
and 
$$F=\sum_{(i,j)\neq (n,n)}A_{nn}^{-1}A_{ij}v_{ij}+\frac1t(A_{nn}^{-1}{P}-p)v_t+\frac1{t^2}(A_{nn}^{-1}{Q}-q)v
+A_{nn}^{-1}N.$$
It is easy to see that $F$ can be viewed as a function in $y', t$ and 
\be\label{gtt-Coefficients}
\frac{v}{t}, v_t, \frac{D_{x'} v}{t}, D_{x'}v_t, 
D^2_{x'}v, \frac{v^2}{t^3},\frac{vv_t}{t^2}, \frac{v_t^2}{t},\ee
and $F$ is linear in the last three quantities in \eqref{gtt-Coefficients}. 
Throughout this section, we assume that  
{\it $p$ and $q$ are constants and $F$ is smooth in all of its arguments except $y'$.} 
In the following, we denote by $\prime$ the derivative with respect to $t$. This 
should not be confused with $y'$, the first $n-1$ coordinates of the point. 

We assume that $\underline{m}$ and $\overline{m}$ are two constants such that 
\be\label{eq-Assumption_m1}
p=1-(\underline{m}+\overline{m}), \quad q=\underline{m}\cdot\overline{m},\ee
and 
\be\label{eq-Assumption_m2}\underline{m}\le 0,\, \overline{m}\ge 3.\ee
Hence, $t^{\underline{m}}$ and $t^{\overline{m}}$ are 
solutions of the linear homogeneous equation corresponding to \eqref{gtt}. 
We also assume that $\overline{m}$ is an integer and demonstrate that 
the higher regularity fails due to the presence of $\log t$. 

We now treat $F$ in \eqref{gtt} as a function of $y'$ and $t$. 
A standard calculation yields the following result: {\it Let $v$ be a solution of 
\eqref{gtt} satisfying 
\be\label{eq:r7assumption} 
t^{-\underline{m}}v\to 0\quad\text{as }t\to 0.\ee
Then,}
\begin{align}\label{eq:r8}\begin{split}
v(y', t)&=\Big[v(y',r)r^{-\overline{m}}
-\frac{r^{\underline{m}-\overline{m}}}{\overline{m}-\underline{m}}
\int_0^r s^{1-\underline{m}}F(y',s)ds\Big] t^{\overline{m}}\\
&\qquad+\frac{t^{\underline{m}}}{\overline{m}-\underline{m}} 
\int_0^t s^{1-\underline{m}}F(y',s) ds
+\frac{t^{\overline{m}}}{\overline{m}-\underline{m}}
\int_t^r s^{1-\overline{m}}F(y',s) ds.\end{split}
\end{align}
We note that the regularity of $v$ in $y'$ inherits from that of $v(\cdot, r)$ and $F$, 
as long as the integrals in \eqref{eq:r8} make sense. We now rewrite 
\eqref{eq:r8} so we can discuss the regularity of $v$ in $t$. 

We now discuss the optimal regularity of solutions 
up to $C^{\overline{m}-1, \alpha}$. 
This method was adapted from \cite{JianWang2013JDG}. 
Throughout this section, $\tau$ and $\nu$ are nonnegative integers, 
used for the order of differentiation with respect to $y'$ and $t$, respectively. 

We start with the tangential regularity \eqref{eq-BoundTangentialGradient_vvAlpha} 
we proved in the previous section 
and view it as our basic assumption. 

\begin{theorem}\label{thrm-regularityC^{n-1}}
Assume that $\underline{m}$ is a constant and $\overline{m}$ is an integer satisfying 
\eqref{eq-Assumption_m1} and \eqref{eq-Assumption_m2} and 
that $F$ is $C^{\ell-2,\alpha}$ in $y'$ and smooth 
in other arguments, for some $\ell\ge \overline{m}-1$ and
$\alpha\in (0,1)$. 
Let $v\in C^{1,1}(\bar G_R)\cap C^{\ell, \alpha}(G_R)$ 
be a solution of \eqref{gtt} in $G_R$, 
for some $R>0$, and satisfy, 
for any $\tau\le \ell-2$ and  any  $r\in (0, R)$, 
\begin{align}\label{eq-Quasi-TangentialRegu-assumption}
t^{-2}{D^\tau_{y'}v},\, t^{-1}{DD^\tau_{y'}v},\, D^2D^\tau_{y'}v
\in C^{\alpha}(\bar{G}_r).
\end{align}
Then, for any $\nu\le \overline{m}-3$, any $\tau\le \ell-2-\nu$, and any  $r\in (0,R)$, 
\begin{align}\label{eq-Quasi-TangentialRegu-conclusion}
t^{-2}{D^\tau_{y'}\partial_t^\nu v},\, t^{-1}{DD^\tau_{y'}\partial_t^\nu v},\, D^2D^\tau_{y'}\partial_t^\nu v
\in C^{\alpha}(\bar{G}_r).
\end{align}
In particular, for any $\tau\le\ell-\overline{m}+1$, 
$$D_{y'}^\tau v\in C^{\overline{m}-1, \alpha}(\bar G_r).$$
\end{theorem}

\begin{proof} If $\overline{m}= 3$, then \eqref{eq-Quasi-TangentialRegu-conclusion} is the same as 
\eqref{eq-Quasi-TangentialRegu-assumption}. 
We consider $\overline{m}\ge 4$ and fix an $r\in (0,R)$. We introduce a new function $w$ by 
\begin{equation}\label{eq-relation-v-w}w=\frac{v}{t^2},\end{equation} 
and consider 
\begin{align}\label{eq-Quasi-TangentialRegu-conclusion-equiv}
D^\tau_{y'}\partial_t^\nu w,\, tDD^\tau_{y'}\partial_t^\nu w,\, 
t^2D^2D^\tau_{y'}\partial_t^\nu w
\in C^{\alpha}(\bar{G}_r).
\end{align}
By \eqref{eq-Quasi-TangentialRegu-assumption}, we have \eqref{eq-Quasi-TangentialRegu-conclusion-equiv} 
for $\nu=0$ and any $\tau\le \ell-2$. 
We will prove \eqref{eq-Quasi-TangentialRegu-conclusion-equiv}
for any $\nu\le \overline{m}-3$ and $\tau\le \ell-2-\nu$. 
Then, \eqref{eq-Quasi-TangentialRegu-conclusion} holds for any $\nu\le \overline{m}-3$ and $\tau\le \ell-2-\nu$.

We will prove \eqref{eq-Quasi-TangentialRegu-conclusion-equiv} by an induction on  $\nu$. 
Fix an integer $k$ with $1\le k\le \overline{m}-3$. 
We assume 
\eqref{eq-Quasi-TangentialRegu-conclusion-equiv}
for any nonnegative integers  $\nu\le k-1$ and $\tau\le \ell-2-\nu$, 
and proceed to prove
\eqref{eq-Quasi-TangentialRegu-conclusion-equiv}  for  $\nu\le k$
and $\tau\le \ell-2-\nu$. 

We claim, for any nonnegative integers $\nu\le k$ and $\tau\le \ell-2-\nu$,  
\begin{equation}\label{eq-regularity-F}\partial^\nu_tD_{x'}^\tau F\in C^\alpha(\bar G_{r}).\end{equation}
Recall that $F$ is a function of $y', t$, and quantities in \eqref{gtt-Coefficients}. 
In terms of $w$, $F$ is a function of $y',t$, and 
\begin{equation}\label{gtt-Coefficients-w}
t^2D_{y'}^2w, t^2D_{y'}w_t, t^2w_t, tD_{y'}w, tw, tw^2, t^2ww_t, t^3w_t^2.
\end{equation}
We need to calculate $\partial_t^\nu D_{y'}^\tau$ acting on these quantities, 
for $\nu\le k$ and $\tau\le \ell-2-\nu$. 
To this end, we divide quantities in \eqref{gtt-Coefficients-w} into two groups, 
the first group consisting of the first five quantities and the second group the last three. 
For the first group, we consider $t^2\partial_{t}D_{x'}w$ for an illustration. 
Note, for any nonnegative integers $\nu$ and $\tau$, 
\begin{align*}
\partial^\nu_tD_{y'}^\tau(t^2\partial_{t}D_{y'}w)=t^2\partial_t^2\partial_{t}^{\nu-1}D_{y'}^{\tau+1}w
+2\nu t\partial_t\partial_{t}^{\nu-1} D_{y'}^{\tau+1}w+\nu(\nu-1)\partial_{t}^{\nu-1} D_{y'}^{\tau+1}w.\end{align*}
We intentionally write derivatives with respect to $t$ as above. 
By the induction hypothesis, \eqref{eq-Quasi-TangentialRegu-conclusion-equiv} holds
for any nonnegative integers $\nu\le k-1$ and $\tau\le \ell-2-\nu$. Hence, 
for any nonnegative integers  $\nu\le k$ and $\tau\le \ell-2-\nu$, 
$$\partial^\nu_tD_{x'}^\tau(t^2\partial_{t}D_{x'}u)\in C^\alpha(\bar G_{r}).$$
A similar result holds for $t^2D_{y'}^2w, t^2w_t, tD_{y'}w$, and $tw$. 
For the last three quantities in \eqref{gtt-Coefficients-w}, we have 
\begin{align*} 
\partial_t(tw^2)&=2w\cdot tw_t+w^2,\\
\partial_t(t^2ww_t)&=t^2w_{tt}+(tw_t)^2+2w\cdot tw_t,\\
\partial_t(t^3w_t^2)&=2tw_t\cdot t^2w_{tt}+3(tw_t)^2.
\end{align*}
In other words, $\partial_t(tw^2), \partial_t(t^2ww_t)$, and 
$\partial_t(t^3w_t^2)$ can be expressed in terms of $w, tw_t$, and $t^2w_{tt}$. 
Hence, 
for any nonnegative integers  $\nu\le k$ and $\tau\le \ell-2-\nu$, 
$$\partial^\nu_tD_{y'}^\tau(tw^2), 
\partial^\nu_tD_{y'}^\tau(t^2ww_t), 
\partial^\nu_tD_{y'}^\tau(t^3w_t^2)\in C^\alpha(\bar G_{r}).$$
As a consequence, we obtain \eqref{eq-regularity-F} 
for any nonnegative integers $\nu\le k$ and $\tau\le \ell-2-\nu$.

In the following, we suppress the dependence of functions on $y'$ and 
write $v(t)$ instead of $v(y',t)$. 
We write \eqref{eq:r8} as 
\begin{align}\label{eq-expression-v}
v(t)=c_{\overline{m}} t^{\overline{m}}
+\underline{v}(t)
+\overline{v}(t),
\end{align}
where 
\begin{align}\label{eq-expression-c-m}c_{\overline{m}} =v(r)r^{-\overline{m}}
-\frac{r^{\underline{m}-\overline{m}}}{\overline{m}-\underline{m}}
\int_0^r s^{1-\underline{m}}F(s)ds,\end{align}
and 
\begin{align}\label{eq-expression-under-over-v}\begin{split} 
\underline{v}(t)&=\frac{t^{\underline{m}}}{\overline{m}-\underline{m}} 
\int_0^t s^{1-\underline{m}}F(s) ds,\\
\overline{v}(t)&=\frac{t^{\overline{m}}}{\overline{m}-\underline{m}}
\int_t^r s^{1-\overline{m}}F(s) ds.\end{split}\end{align}
By dividing \eqref{eq-expression-v} by $t^2$, we obtain 
\begin{align}\label{eq-expression-w-rewrite}
w(t)=c_{\overline{m}} t^{\overline{m}-2}
+\underline{w}(t)
+\overline{w}(t),
\end{align}
where
\begin{align}\label{eq-expression-under-over-w}\begin{split} 
\underline{w}(t)&=\frac{t^{\underline{m}-2}}{\overline{m}-\underline{m}} 
\int_0^t s^{1-\underline{m}}F(s) ds,\\
\overline{w}(t)&=\frac{t^{\overline{m}-2}}{\overline{m}-\underline{m}}
\int_t^r s^{1-\overline{m}}F(s) ds.\end{split}\end{align}
A straightforward computation yields 
\begin{align}\label{eq-expression-tw-t}
t\partial_tw(t)= (\overline{m}-2)c_{\overline{m}}t^{\overline{m}-2}
+(\underline{m}-2)\, \underline{w}(t)
+(\overline{m}-2)\,\overline{w}(t),
\end{align}
and 
\begin{align}\label{eq-expression-t2w-t2}\begin{split}
t^2\partial^2_{t}w(t)&= (\overline{m}-2)(\overline{m}-3)c_{\overline{m}}t^{\overline{m}}-F(t)\\
&\qquad +(\underline{m}-2)(\underline{m}-3) \underline{w}(t)
+(\overline{m}-2)(\overline{m}-3)\overline{w}(t).
\end{split}\end{align}
By  applying Lemma \ref{lemma-BasicHolderRegularity}  
to $\underline{w}$ and 
applying Lemma \ref{lemma-BasicHolderRegularity2} 
to $\overline{w}$,
we obtain, 
for any nonnegative integers $\nu\le k$ and $\tau\le \ell-2-\nu$, 
\begin{equation}\label{eq-regularity-w}
\partial_t^\nu D^\tau_{y'}\underline{w},\, \partial_t^\nu D^\tau_{y'}\overline{w}
\in C^{\alpha}(\bar G_r).\end{equation}
Note $c_{\overline{m}}\in C^{\ell-2,\alpha} (\bar B_r')$, by \eqref{eq-expression-c-m}. 
By the explicit expressions given by \eqref{eq-expression-w-rewrite}, \eqref{eq-expression-tw-t}, and 
\eqref{eq-expression-t2w-t2}, we have, 
for any nonnegative integers $\nu\le k$ and $\tau\le \ell-2-\nu$,
\begin{equation*}
\partial_t^\nu D^\tau_{y'}w,\, \partial_t^\nu D^\tau_{y'}(t\partial_tw),\, 
\partial_t^\nu D^\tau_{y'}(t^2\partial^2_tw)\in C^{\alpha}(\bar G_r).\end{equation*}
This implies \eqref{eq-Quasi-TangentialRegu-conclusion-equiv},  for any $\nu\le k$
and $\tau\le \ell-2-\nu$. 
\end{proof}

Next, we discuss the higher regularity of solutions. 
In the proof of Theorem \ref{thrm-regularityC^{n-1}}, 
the maximal $k$ 
is $\overline{m}-3$. In the following, we will calculate $F$ and $v$ inductively by 
increasing $k$ and obtain an expression of $v$ accordingly for each 
large $k$. We first consider the case $k=\overline{m}-2$. 

\begin{lemma}\label{lemma-Main-UpTo-n}
Assume that $\underline{m}$ is a constant and $\overline{m}$ is an integer satisfying 
\eqref{eq-Assumption_m1} and \eqref{eq-Assumption_m2} and 
that $F$ is $C^{\ell-2,\alpha}$ in $y'$ and smooth 
in other arguments,  for some $\ell\ge \overline{m}$ and
$\alpha\in (0,1)$. 
Let $v\in C^{1,1}(\bar G_R)\cap C^{\ell, \alpha}(G_R)$ 
be a solution of \eqref{gtt} in $G_R$, 
for some $R>0$, satisfying 
\eqref{eq-Quasi-TangentialRegu-assumption}  for any $r\in (0, R)$ and any 
$\tau\le\ell-2$. 
Then, 
\begin{align}\label{eq-expression_g_n1}
v=\sum_{i=2}^{\overline{m}-1}c_i t^i
+c_{\overline{m},1}t^{\overline{m}}\log t+c_{\overline{m},0}t^{\overline{m}}
+R_{\overline{m}}\quad\text{in }G_R,
\end{align}
where $c_{i}$, for $2\le i\le \overline{m}-1$, and $c_{\overline{m},j}$, for $0\le j\le 1$, are 
functions in $B_{R}'$ with 
$c_i\in C^{\ell-i,\alpha}(B'_R)$, for $2\le i\le \overline{m}-1$, 
$c_{\overline{m},1}\in C^{\ell-\overline{m},\alpha}(B'_R)$, and 
$c_{\overline{m},0}\in C^{\ell-\overline{m},\epsilon}(B'_R)$, for any $\epsilon\in (0,\alpha)$,
and $R_{\overline{m}}$ is a function in $G_R$ such that, 
for any $\nu\le \overline{m}$, $\tau\le \ell-\overline{m}$, any $r\in (0,R)$, and any $\epsilon\in (0,\alpha)$, 
\begin{align}\label{eq-regularity_h_n1}
D_{y'}^\tau\partial_t^\nu R_{\overline{m}}
\in C^{\overline{m}-\nu, \epsilon}(\bar{G}_r), 
\end{align}
and 
\be\label{eq-regularity_h_n2}
|D_{y'}^\tau\partial_t^\nu R_{\overline{m}}|
\le Ct^{\overline{m}-\nu+\alpha}\quad\text{in }G_r.\ee
\end{lemma}


\begin{proof} 
We adopt notations in the proof of Theorem \ref{thrm-regularityC^{n-1}} and fix an $r\in (0, R)$.  
Recall the function $w$ introduced in \eqref{eq-relation-v-w}. 
In the proof of Theorem \ref{thrm-regularityC^{n-1}}, we verified \eqref{eq-Quasi-TangentialRegu-conclusion-equiv}
for any $\nu\le \overline{m}-3$ and $\tau\le \ell-2-\nu$, i.e., 
\begin{align*}
D^\tau_{y'}\partial_t^\nu w,\, tDD^\tau_{y'}\partial_t^\nu w,\, 
t^2D^2D^\tau_{y'}\partial_t^\nu w
\in C^{\alpha}(\bar{G}_r).
\end{align*}
By  a similar argument as in the proof of 
Theorem \ref{thrm-regularityC^{n-1}}, we get, 
for any $\nu\le \overline{m}-2$ and $\tau\le \ell-2-\nu$, 
\begin{align}\label{eq-regularity-F1-int}D^\tau_{y'}\partial_t^\nu F\in C^{\alpha}(\bar G_r).\end{align}
Set 
$$a_i=\frac{1}{i!}\partial_t^iF(\cdot,0)\quad\text{for }i=0, \cdots, \overline{m}-2,$$
and write 
\begin{equation}\label{eq-identity-F-int-overline-m}F=\sum_{i=0}^{\overline{m}-2} a_it^i
+S_{\overline{m}-2}.\end{equation}
By \eqref{eq-regularity-F1-int}, we have 
\begin{equation}\label{eq-regularity-a-lower} 
a_i\in C^{\ell-2-i, \alpha}(B_1')\quad\text{for any }i=0, \cdots, \overline{m}-2,\end{equation}
and,  
for any $\nu\le \overline{m}-2$ and $\tau\le \ell-\overline{m}$, 
\begin{align}\label{eq-regularity-S-m}\begin{split}
\partial_t^\nu D_{y'}^\tau S_{\overline{m}-2}&\in C^\alpha(\bar G_r),\\
|\partial_t^\nu D_{y'}^\tau S_{\overline{m}-2}|&\le Ct^{\overline{m}-2-\nu+\alpha}
\quad\text{in } G_{r}. 
\end{split}\end{align}

By substituting \eqref{eq-identity-F-int-overline-m} in \eqref{eq:r8} 
and a straightforward computation, we have 
\begin{align}\label{eq-indentity-u-int-overline-m}
v=\sum_{i=2}^{\overline{m}-1}c_it^i+c_{\overline{m},1}t^{\overline{m}}\log t+c_{\overline{m},0}t^{\overline{m}}
+R_{\overline{m}},
\end{align}
where 
\begin{align}\label{eq-expression-coefficient-m1-int}\begin{split}
c_{i}&=-\frac{a_{i-2}}{(i-\underline{m})(i-\overline{m})}\quad\text{for }i=2, \cdots, \overline{m}-1,\\
c_{\overline{m},1}&=-\frac{a_{\overline{m}-2}}{\overline{m}-\underline{m}},\\
c_{\overline{m}, 0}&=v(\cdot,r)r^{-\overline{m}}+\sum_{i=0}^{\overline{m}-1}
\frac{a_{i-2}r^{i-\overline{m}}}{(i-\underline{m})(i-\overline{m})}
+\frac{a_{\overline{m}-2}\log r}{\overline{m}-\underline{m}}\\
&\qquad 
-\frac{r^{\underline{m}-\overline{m}}}{\overline{m}-\underline{m}}
\int_0^r s^{1-\underline{m}}S_{\overline{m}-2}ds
+\frac{1}{\overline{m}-\underline{m}}\int_0^rs^{1-\overline{m}}S_{\overline{m}-2}ds,
\end{split}\end{align}
and
\begin{align}\label{eq-expression-remainder-m-int}
R_{\overline{m}}=\frac{t^{\underline{m}}}{\overline{m}-\underline{m}}\int_0^ts^{1-\underline{m}}S_{\overline{m}-2}ds
-\frac{t^{\overline{m}}}{\overline{m}-\underline{m}}\int_0^ts^{1-\overline{m}}S_{\overline{m}-2}ds.
\end{align}
We note that $c_i$ is computed in terms of $a_{i-2}$ at the same boundary point, 
for $i=2,\cdots, \overline{m}-1$, and 
$c_{\overline{m},1}$ 
in terms of $a_{\overline{m}-2}$ at the same boundary point.
However, the expression for $c_{\overline{m},0}$ is more complicated
and involves the solution $v$ and its integrals. 

By \eqref{eq-regularity-a-lower}, we have 
$c_i\in C^{\ell-i,\alpha}(B'_1)$, for $2\le i\le \overline{m}-1$, and
$c_{\overline{m},1}\in C^{\ell-\overline{m}, \alpha}(B'_1)$.
In addition, by \eqref{eq-regularity-S-m} and Lemma \ref{lemma-BasicHolderRegularity1-int0}, we have 
$c_{\overline{m},0}\in C^{\ell-\overline{m},\epsilon}(B'_R)$ for any $\epsilon\in (0,\alpha)$. 
To analyze the regularity of $R_{\overline{m}}$, it is convenient to introduce $T_{\overline{m}-2}$ by 
\begin{equation}\label{eq-relation-R-T}
R_{\overline{m}}=t^2T_{\overline{m}-2}.
\end{equation} 
By \eqref{eq-expression-remainder-m-int}, we have 
$$T_{\overline{m}-2}=\underline{T}_{\overline{m}-2}+\overline{T}_{\overline{m}-2},$$
where 
\begin{align*}
\underline{T}_{\overline{m}-2}
&=\frac{t^{\underline{m}-2}}{\overline{m}-\underline{m}}\int_0^ts^{1-\underline{m}}S_{\overline{m}-2}ds,\\
\overline{T}_{\overline{m}-2}
&=-\frac{t^{\overline{m}-2}}{\overline{m}-\underline{m}}\int_0^ts^{1-\overline{m}}S_{\overline{m}-2}ds.
\end{align*}
A simple computation yields
\begin{align*}
t\partial_tT_{\overline{m}-2}= 
(\underline{m}-2) \underline{T}_{\overline{m}-2}
+(\overline{m}-2)\overline{T}_{\overline{m}-2},
\end{align*}
and 
\begin{align*}
t^2\partial^2_{t}T_{\overline{m}-2}= S_{\overline{m}-2}
+(\underline{m}-2)(\underline{m}-3) \underline{T}_{\overline{m}-2}
+(\overline{m}-2)(\overline{m}-3)\overline{T}_{\overline{m}-2}.
\end{align*}
By \eqref{eq-regularity-S-m},  Lemma \ref{lemma-BasicHolderRegularity}, 
and Lemma \ref{lemma-BasicHolderRegularity1}, we have, 
for any $\nu\le \overline{m}-2$ and $\tau\le \ell-\overline{m}$, 
\begin{align*}
\partial_t^\nu D^\tau_{y'}\underline{T}_{\overline{m}-2}\in C^{\alpha}(\bar G_r),
\quad 
\partial_t^\nu D^\tau_{y'}\overline{T}_{\overline{m}-2}\in C^{\epsilon}(\bar G_r),\end{align*}
and 
\begin{align*}
|\partial_t^\nu D^\tau_{y'}\underline{T}_{\overline{m}-2}|
+|\partial_t^\nu D^\tau_{y'}\overline{T}_{\overline{m}-2}|\le Ct^{\overline{m}-2-\nu+\alpha}\quad\text{in }G_{r}.\end{align*}
As a consequence, we obtain, 
for any $\nu\le \overline{m}-2$ and $\tau\le \ell-\overline{m}$, 
\begin{align}\label{eq-regularity-Rm-bar-equiv}
\partial_t^\nu D_{y'}^\tau T_{\overline{m}-2}, 
\partial_t^\nu D_{y'}^\tau(t\partial_tT_{\overline{m}-2}), 
\partial_t^\nu D_{y'}^\tau(t^2\partial^2_{t}T_{\overline{m}-2})\in C^{\epsilon}(\bar{G}_r), 
\end{align}
and
\begin{align}\label{eq-decay-Rm-bar-equiv}
&|\partial_t^\nu D_{y'}^\tau T_{\overline{m}-2}|
+|\partial_t^\nu D_{y'}^\tau(t\partial_tT_{\overline{m}-2})|
+|\partial_t^\nu D_{y'}^\tau(t^2\partial^2_{t}T_{\overline{m}-2})|
\le Ct^{\overline{m}-2-\nu+\alpha}
\quad\text{in }G_{r}.
\end{align}
Back to $R_{\overline{m}}$, \eqref{eq-relation-R-T}, \eqref{eq-regularity-Rm-bar-equiv}, and \eqref{eq-decay-Rm-bar-equiv} 
imply, for any $\nu\le \overline{m}-2$ and $\tau\le \ell-\overline{m}$, 
\begin{align*}
\partial_t^\nu D_{x'}^\tau(t^{-2} R_{\overline{m}}), 
\partial_t^\nu D_{x'}^\tau\big(t^{-1}\partial_tR_{\overline{m}}), 
\partial_t^\nu D_{x'}^\tau(\partial^2_{t}R_{\overline{m}})\in C^{\epsilon}(\bar{G}_r), 
\end{align*}
and
\begin{align*}
|\partial_t^\nu D_{x'}^\tau(t^{-2} R_{\overline{m}})|+
|\partial_t^\nu D_{x'}^\tau\big(t^{-1}\partial_tR_{\overline{m}})|+
|\partial_t^\nu D_{x'}^\tau(\partial^2_{t}R_{\overline{m}})|\le Ct^{\overline{m}-2-\nu+\alpha}
\quad\text{in }G_{r}.\end{align*}
Therefore, we conclude \eqref{eq-regularity_h_n1}-\eqref{eq-regularity_h_n2}. 
\end{proof}

In Lemma \ref{lemma-Main-UpTo-n}, the coefficients $c_2, \cdots, c_{\overline{m}-1}$, and 
$c_{\overline{m}, 1}$ have explicit expressions in terms of $p, q$, and $F$ and hence their regularity 
can be determined by that of $F$. However, no such expression exists for $c_{\overline{m}, 0}$
and there is a slight loss of regularity for $c_{\overline{m},0}$. 
Under the assumption $c_{\overline{m}, 1}=0$, Lin \cite{Lin1989Invent} and Tonegawa
\cite{Tonegawa1996MathZ} claimed $u\in C^{\overline{m}, \alpha}(\bar B_r)$, 
for the solution $u$ of the equation \eqref{eq-Intro-Equ} with $\overline{m}=n+1$. Their proof actually yields 
$u\in C^{\overline{m}, \epsilon}(\bar B_r)$, for any $\epsilon\in (0,\alpha)$. 

We now expand $v$ to arbitrary orders and estimate remainders. 

\begin{theorem}\label{Thm-MainThm}
Assume that $\underline{m}$ is a constant and $\overline{m}$ is an integer satisfying 
\eqref{eq-Assumption_m1} and \eqref{eq-Assumption_m2} and 
that $F$ is $C^{\ell-2,\alpha}$ in $y'$ and smooth in other 
arguments,  for some integers $\ell\ge k\ge\overline{m}$ and
some $\alpha\in (0,1)$. 
Let $v\in C^{1,1}(\bar G_R)\cap C^{\ell, \alpha}(G_R)$ 
be a solution of \eqref{gtt} in $G_R$, 
for some $R>0$, satisfying 
\eqref{eq-Quasi-TangentialRegu-assumption}  for any $r\in (0, R)$ and any 
$\tau\le \ell-2$. 
Then, 
\begin{align}\label{eq-expression_g_n2z}
v=\sum_{i=2}^{\overline{m}-1}c_i t^i
+\sum_{i=\overline{m}}^k \sum_{j=0}^{N_i} c_{i, j}t^{i}(\log t)^j
+R_{k}\quad\text{in }G_R,\end{align}
where $c_{i}$ and $c_{i,j}$ are 
functions in $B_{R}'$ with 
$c_i\in C^{\ell-i,\alpha}(B'_R)$, for $i=2, \cdots, \overline{m}-1$, 
$c_{\overline{m},1}\in C^{\ell-\overline{m},\alpha}(B'_R)$, and 
$c_{i,j}\in C^{\ell-i,\epsilon}(B'_R)$, for $(i,j)=(\overline{m}, 0)$ or 
$i>\overline{m}$, and for any $\epsilon\in (0,\alpha)$,
$N_i$ is a nonnegative integer for $i\ge \overline{m}$, with $N_{\overline{m}}=1$, 
and $R_k$ is a function in $G_R$ such that, for 
any $\nu\le k$,  $\tau\le \ell-k$, any $r\in (0, R)$, and any $\epsilon\in (0,\alpha)$, 
\begin{align}\label{eq-regularity_R_k1z}
D_{y'}^\tau\partial_t^\nu R_k 
\in C^{k-\nu, \epsilon}(\bar{G}_r), 
\end{align}
and 
\be\label{eq-regularity_R_k2z}
|D_{y'}^\tau\partial_t^\nu R_k|
\le Ct^{k-\nu+\alpha}\quad\text{in }G_r.\ee
\end{theorem}

\begin{proof} 
We adopt notations in the proofs of Theorem \ref{thrm-regularityC^{n-1}}
and Lemma \ref{lemma-Main-UpTo-n} and fix an $r\in (0, R)$.  
Recall the function $w$ introduced in \eqref{eq-relation-v-w}. 
In the following, we will expand $v, w$, and $F$ according to  $t^i$ and $t^i(\log t)^j$. 
We always denote by $c_i, c_{i,j}$, and $R_k$ the coefficients and the remainder for $v$, 
by $b_i, b_{i,j}$, and $T_k$ for $w$, and 
by $a_i, a_{i,j}$, and $S_k$ for $F$. We divide the proof into several steps.

{\it Step 1. The setup.} 
For any $k$ with $\ell\ge k\ge\overline{m}$, we write \eqref{eq-expression_g_n2z} in terms of $w$ as 
\begin{align}\label{eq-expression-w-general-k}
w=\sum_{i=0}^{\overline{m}-3}b_i t^i
+\sum_{i=\overline{m}-2}^{k-2} \sum_{j=0}^{M_i} b_{i, j}t^{i}(\log t)^j
+T_{k-2}\quad\text{in }G_R,
\end{align}
where $b_{i}$ and $b_{i,j}$ are 
functions in $B_{R}'$ 
and $T_{k-2}$ is a function in $G_R$ given by 
\begin{align}\label{eq-relation-c-a-k}\begin{split}
b_i&=c_{i+2}\quad\text{for }0\le i\le \overline{m}-3,\\
b_{i,j}&=c_{i+2, j}\quad\text{for }\overline{m}-2\le i\le k-2\text{ and }0\le j\le M_i\equiv N_{i+2},
\end{split}\end{align}
and 
\begin{equation}\label{eq-relation-R-T-k}
T_{k-2}=t^{-2}R_k.
\end{equation}
We will prove \eqref{eq-expression-w-general-k} and, for any $\epsilon\in (0,\alpha)$,  
\begin{align}\label{eq-regularity-b-k}\begin{split}
&b_i\in C^{\ell-2-i,\alpha}(B'_R)\text{ for }0\le i\le \overline{m}-3, \quad
b_{\overline{m}-2,1}\in C^{\ell-\overline{m},\alpha}(B'_R),\\
&b_{i,j}\in C^{\ell-2-i,\epsilon}(B'_R)\text{ for }(i,j)=(\overline{m}-2, 0)\text{ or }
\overline{m}-1\le i\le k-2, 0\le j\le M_i,
\end{split}\end{align} 
and, for any $\nu\le k-2$ and $\tau\le \ell-k$, 
\begin{align}\label{eq-regularity-Rk-equiv}\begin{split}
\partial_t^\nu D_{y'}^\tau T_{k-2}, 
\partial_t^\nu D_{y'}^\tau(t\partial_tT_{k-2}), 
\partial_t^\nu D_{y'}^\tau(t^2\partial^2_{t}T_{k-2})&\in C^{\epsilon}(\bar{G}_r), \\
|\partial_t^\nu D_{y'}^\tau T_{k-2}|
+|\partial_t^\nu D_{y'}^\tau(t\partial_tT_{k-2})|
+|\partial_t^\nu D_{y'}^\tau(t^2\partial^2_{t}T_{k-2})| &\le Ct^{k-2-\nu+\alpha}
\quad\text{in }G_{r}.
\end{split}\end{align}
Back to $R_k$,  \eqref{eq-relation-R-T-k} and \eqref{eq-regularity-Rk-equiv} 
imply, 
for any $\nu\le k-2$ and $\tau\le \ell-k$, 
\begin{align*}
\partial_t^\nu D_{y'}^\tau(t^{-2} R_{k}), 
\partial_t^\nu D_{y'}^\tau\big(t^{-1}\partial_tR_{k}), 
\partial_t^\nu D_{y'}^\tau(\partial^2_{t}R_{k})\in C^{\epsilon}(\bar{G}_r), 
\end{align*}
and
\begin{align*}
|\partial_t^\nu D_{y'}^\tau(t^{-2} R_{k})|+
|\partial_t^\nu D_{y'}^\tau\big(t^{-1}\partial_tR_{k})|+
|\partial_t^\nu D_{y'}^\tau(\partial^2_{t}R_{k})|\le Ct^{k-2-\nu+\alpha}
\quad\text{in }G_{r}.\end{align*}
Therefore, we conclude \eqref{eq-expression_g_n2z} and \eqref{eq-regularity_R_k1z}-\eqref{eq-regularity_R_k2z}
for any $\nu\le k$ and  $\tau\le \ell-k$. 

We prove \eqref{eq-expression-w-general-k} and 
\eqref{eq-regularity-b-k}-\eqref{eq-regularity-Rk-equiv} 
by an induction on $k$. 
First, \eqref{eq-expression-w-general-k} and 
\eqref{eq-regularity-b-k}-\eqref{eq-regularity-Rk-equiv} 
hold  for $k=\overline{m}$ by Lemma \ref{lemma-Main-UpTo-n}. 
In particular, \eqref{eq-regularity-Rk-equiv} 
with $k=\overline{m}$ is simply 
\eqref{eq-regularity-Rm-bar-equiv}-\eqref{eq-decay-Rm-bar-equiv}. 
We fix an integer $k$ with $\overline{m}<k\le \ell$. We assume that 
\eqref{eq-expression-w-general-k} and 
\eqref{eq-regularity-b-k}-\eqref{eq-regularity-Rk-equiv} 
hold with $k$ replaced by 
any integer between $\overline{m}$ and $k-1$ 
and then proceed to prove \eqref{eq-expression-w-general-k} and 
\eqref{eq-regularity-b-k}-\eqref{eq-regularity-Rk-equiv} 
for $k$. 
With $b_{i}$, for $0\le i\le \overline{m}-3$, and $b_{i,j}$, for $\overline{m}-2\le i\le k-3$ and $0\le j\le M_i$, already determined, 
we will find $b_{k-2,j}$, for $0\le j\le M_{k-2}$, and $T_{k-2}$, and prove that they have the stated properties. 
In the following, we will use \eqref{eq-regularity-Rk-equiv} 
with 
$k$ replaced by $k-1$ and $k-2$. 

{\it Step 2. The expansion of $F$.}
We now analyze $F$ and claim 
\begin{equation}\label{eq-identity-F-int-k}F=\sum_{i=0}^{\overline{m}-2} a_it^i
+\sum_{i=\overline{m}-1}^{k-2}\sum_{j=0}^{M_i} a_{i,j}t^i(\log t)^j+S_{k-2}\quad\text{in }G_R,\end{equation}
where $a_i$ and $a_{i,j}$ are functions on $B_R'$ 
and $S_{k-2}$ a function in $G_R$ 
such that, for any $\epsilon\in (0,\alpha)$,  
\begin{align}\label{eq-regularity-coefficients-F-k}\begin{split}
&a_i\in C^{\ell-2-i,\alpha}(B'_R)\quad\text{for }0\le i\le\overline{m}-2,\\  
&a_{i, j}\in C^{\ell-2-i,\epsilon}(B'_R)\quad\text{for } 
\overline{m}-1\le i\le k-2\text{ and }0\le j\le M_i,
\end{split}\end{align}
and, for any nonnegative integers $\nu\le k-2$ 
and $\tau\le \ell-k$,  
\begin{align}\label{eq-LinearRegularity-F-m1}\begin{split}
\partial_t^\nu D_{y'}^\tau S_{k-2}&\in C^{\epsilon}(\bar{G}_r),\\
|\partial_t^\nu D_{y'}^\tau S_{k-2}|&\le  Ct^{k-2-\nu+\alpha}
\quad\text{in }G_{r}.\end{split}\end{align}
We point out that the second summation in \eqref{eq-identity-F-int-k} for $t^i(\log t)^j$ starts from $i=\overline{m}-1$. 
The regularity of $a_i$ for $0\le i\le  \overline{m}-2$ is already established in \eqref{eq-regularity-a-lower}. 
The regularity of $a_{i,j}$ for $i=\overline{m}-1, \cdots, k-3$ is part of the induction hypothesis. 
We will focus on $a_{k-2,j}$, for $0\le j\le M_{k-2}$, and $S_{k-2}$. 

Recall that $F$ is a function of $y', t$, and the quantities given by \eqref{gtt-Coefficients-w}, i.e., 
\begin{equation*} 
t^2D_{y'}^2w, t^2D_{y'}w_t, t^2w_t, tD_{y'}w, tw, tw^2, t^2ww_t, t^3w_t^2.
\end{equation*}
By the induction hypothesis, \eqref{eq-expression-w-general-k} and 
\eqref{eq-regularity-Rk-equiv} 
hold for $k-2$ and $k-1$ replacing $k$. 
We will use \eqref{eq-expression-w-general-k} and 
\eqref{eq-regularity-Rk-equiv} 
with $k-2$ to analyze $t^2D_{y'}^2w$ 
and use \eqref{eq-expression-w-general-k} and \eqref{eq-regularity-Rk-equiv} 
with $k-1$ to analyze the rest of the 
quantities $t^2D_{y'}w_t, t^2w_t, tD_{y'}w, tw, tw^2, t^2ww_t, t^3w_t^2$.

By \eqref{eq-expression-w-general-k} with $k$ replaced by $k-2$, we have 
\begin{align*}
w=\sum_{i=0}^{\overline{m}-3}b_i t^i
+\sum_{i=\overline{m}-2}^{k-4} \sum_{j=0}^{M_i} b_{i, j}t^{i}(\log t)^j
+T_{k-4}\quad\text{in }G_R,
\end{align*}
where $b_i$ and $b_{i,j}$ satisfy \eqref{eq-regularity-b-k} and $T_{k-4}$ 
satisfies, for any $\nu\le k-4$ and $\tau\le \ell-k$, 
\begin{align*}
\partial_t^\nu D_{y'}^\tau T_{k-4}, 
\partial_t^\nu D_{y'}^\tau(t\partial_tT_{k-4}), 
\partial_t^\nu D_{y'}^\tau(t^2\partial^2_{t}T_{k-4})&\in C^{\epsilon}(\bar{G}_r), \\
|\partial_t^\nu D_{y'}^\tau T_{k-4}|
+|\partial_t^\nu D_{y'}^\tau(t\partial_tT_{k-4})|
+|\partial_t^\nu D_{y'}^\tau(t^2\partial^2_{t}T_{k-4})| &\le Ct^{k-4-\nu+\alpha}
\quad\text{in }G_{r}.
\end{align*}
It is easy to see that 
\begin{align*}
t^2D_{y'}^2w=\sum_{i=2}^{\overline{m}-1}D_{y'}^2b_{i-2} t^i
+\sum_{i=\overline{m}}^{k-2} \sum_{j=0}^{M_i} D_{y'}^2b_{i-2, j}t^{i}(\log t)^j
+t^2D_{y'}^2T_{k-4}\quad\text{in }G_R.
\end{align*}
The coefficients of $t^i$ and $t^i(\log t)^j$ and the remainder $t^2D_{y'}^2T_{k-4}$ satisfy 
\begin{align*} 
&D_{y'}^2b_{i-2}\in C^{\ell-2-i,\alpha}(B'_R)\quad\text{for }0\le i\le\overline{m}-1,\\  
&D_{y'}^2b_{i-2, j}\in C^{\ell-2-i,\epsilon}(B'_R)\quad\text{for } 
\overline{m}\le i\le k-2\text{ and }0\le j\le M_i,
\end{align*}
and, for any nonnegative integers $\nu\le k-2$ 
and $\tau\le \ell-k$,  
\begin{align*} 
\partial_t^\nu D_{y'}^\tau (t^2D_{y'}^2T_{k-4})&\in C^{\epsilon}(\bar{G}_r),\\
|\partial_t^\nu D_{y'}^\tau (t^2D_{y'}^2T_{k-4})|&\le  Ct^{k-2-\nu+\alpha}
\quad\text{in }G_{r}. 
\end{align*}
In other words, the coefficients and the remainder in $t^2D_{y'}^2w$ 
satisfy \eqref{eq-regularity-coefficients-F-k} and \eqref{eq-LinearRegularity-F-m1}. 
Had we used \eqref{eq-expression-w-general-k} and 
\eqref{eq-regularity-Rk-equiv} 
for $k-1$ to study $t^2D_{y'}^2w$, 
there would be a loss of regularity in $y'$ from the remainder $t^2D_{y'}^2T_{k-3}$
and such a loss would accumulate upon iterations.

Next, by \eqref{eq-expression-w-general-k} with $k-1$ replaced by $k-2$, we have 
\begin{align*}
w=\sum_{i=0}^{\overline{m}-3}b_i t^i
+\sum_{i=\overline{m}-2}^{k-3} \sum_{j=0}^{M_i} b_{i, j}t^{i}(\log t)^j
+T_{k-3}\quad\text{in }G_R,
\end{align*}
where $b_i$ and $b_{i,j}$ satisfy \eqref{eq-regularity-b-k} and $T_{k-3}$ 
satisfies, for any $\nu\le k-3$ and $\tau\le \ell-k$, 
\begin{align*}
\partial_t^\nu D_{y'}^\tau T_{k-3}, 
\partial_t^\nu D_{y'}^\tau(t\partial_tT_{k-3}), 
\partial_t^\nu D_{y'}^\tau(t^2\partial^2_{t}T_{k-3})&\in C^{\epsilon}(\bar{G}_r), \\
|\partial_t^\nu D_{y'}^\tau T_{k-3}|
+|\partial_t^\nu D_{y'}^\tau(t\partial_tT_{k-3})|
+|\partial_t^\nu D_{y'}^\tau(t^2\partial^2_{t}T_{k-3})| &\le Ct^{k-3-\nu+\alpha}
\quad\text{in }G_{r}.
\end{align*}
By a similar argument, we can prove that 
the coefficients and the remainders in $t^2D_{y'}w_t$, $t^2w_t, tD_{y'}w, tw$ 
satisfy \eqref{eq-regularity-coefficients-F-k} and \eqref{eq-LinearRegularity-F-m1}. 
We now analyze $tw^2, t^2ww_t, t^3w_t^2$. For an illustration, we consider $tw^2$. 
For brevity of notations, we write 
\begin{align*}
w=\sum_{i=0}^{k-3} \sum_{j=0}^{M_i} b_{i, j}t^{i}(\log t)^j
+T_{k-3}\quad\text{in }G_R.
\end{align*}
Here and hereafter, we set $a_{i,0}=a_i$, $b_{i,0}=b_i$ and $M_i=0$, for $0\le i\le \overline{m}-3$. 
A simple computation yields 
\begin{align*}
tw^2&=
\sum_{i_1, i_2=0}^{k-3} \sum_{j_1=0}^{M_{i_1}}\sum_{j_2=0}^{M_{i_2}} b_{i_1, j_1}b_{i_2, j_2}t^{i_1+i_2+1}(\log t)^{j_1+j_2}
\\
&\qquad +2tT_{k-3}\sum_{i=0}^{k-3} \sum_{j=0}^{M_i} b_{i, j}t^{i}(\log t)^j+tT^2_{k-3}\quad\text{in }G_R.
\end{align*}
For any $i, j$ with $1\le i\le 2k-5$ and $0\le j\le 2M_i$, set 
$$\widehat{a}_{i,j}=\sum_{i_1+i_2+1=i} \sum_{j_1+j_2=j} b_{i_1, j_1}b_{i_2, j_2}.$$
Then, we can write 
\begin{align*}
tw^2=
\sum_{i=1}^{k-2}\sum_{j=0}^{M_i} \widehat{a}_{i,j}t^i(\log t)^j+\widehat{S}_{k-2},
\end{align*} 
where 
\begin{align*} 
\widehat{S}_{k-2}=\sum_{i=k-1}^{2k-5}\sum_{j=0}^{M_i} \widehat{a}_{i,j}t^i(\log t)^j
+2tT_{k-3}\sum_{i=0}^{k-3} \sum_{j=0}^{M_i} b_{i, j}t^{i}(\log t)^j+tT^2_{k-3}\quad\text{in }G_R.
\end{align*}
It is straightforward to verify, 
for any $\epsilon\in (0,\alpha)$,  
\begin{align*}
&\widehat{a}_i\in C^{\ell-2-i,\alpha}(B'_R)\quad\text{for }1\le i\le\overline{m}-2,\\  
&\widehat{a}_{i, j}\in C^{\ell-2-i,\epsilon}(B'_R)\quad\text{for } 
\overline{m}-1\le i\le k-2\text{ and }0\le j\le M_i,
\end{align*}
and, for any nonnegative integers $\nu\le k-2$ 
and $\tau\le \ell-k$,  
\begin{align*}
\partial_t^\nu D_{y'}^\tau \widehat{S}_{k-2}&\in C^{\epsilon}(\bar{G}_r),\\
|\partial_t^\nu D_{y'}^\tau \widehat{S}_{k-2}|&\le  Ct^{k-2-\nu+\alpha}
\quad\text{in }G_{r}.\end{align*}
In other words, the coefficients and the remainder in $tw^2$ 
satisfy \eqref{eq-regularity-coefficients-F-k} and \eqref{eq-LinearRegularity-F-m1}.
It is similar to discuss $t^2ww_t$ and $t^3w_t^2$. 

In summary, if we expand each of 
$t^2D_{y'}^2w, t^2D_{y'}w_t, t^2w_t, tD_{y'}w, tw, tw^2, t^2ww_t, t^3w_t^2$ up to $t^{k-2}$, 
then the coefficients and the remainders satisfy \eqref{eq-regularity-coefficients-F-k} and \eqref{eq-LinearRegularity-F-m1}.
Recall that $F$ is smooth in these quantities. Hence, we can write $F$ in \eqref{eq-identity-F-int-k}, 
and the coefficients and the remainder satisfy \eqref{eq-regularity-coefficients-F-k} and \eqref{eq-LinearRegularity-F-m1}.

{\it Step 3. The expansion of $w$.} 
Next, by comparing \eqref{eq-identity-F-int-k} with \eqref{eq-identity-F-int-overline-m}, we have 
$$S_{\overline{m}-2}=\sum_{i=\overline{m}-1}^{k-2}\sum_{j=0}^{M_i}a_{i,j}t^i(\log t)^j+S_{k-2}.$$
Substituting such an expression in \eqref{eq-indentity-u-int-overline-m} and \eqref{eq-expression-remainder-m-int}, we have 
\begin{align}\label{eq-expression-v-k-intermediate}\begin{split}
v&=\sum_{i=2}^{\overline{m}-1}c_it^i+c_{\overline{m},1}t^{\overline{m}}\log t+ c_{\overline{m},0}t^{\overline{m}}\\
&\qquad +\frac{1}{\overline{m}-\underline{m}}\sum_{i=\overline{m}-1}^{k-2}\sum_{j=0}^{M_i}
a_{i,j}\int_0^t [t^{\underline{m}}s^{i+1-\underline{m}}-t^{\overline{m}}s^{i+1-\overline{m}}](\log s)^jds+R_k,
\end{split}\end{align}
where $c_i$, for $2\le i\le \overline{m}-1$, $c_{\overline{m}, 1}$, and $c_{\overline{m}, 0}$ are given by 
\eqref{eq-expression-coefficient-m1-int}, and $R_k$ is given by 
\begin{align*} 
R_{k}=\frac{t^{\underline{m}}}{\overline{m}-\underline{m}}\int_0^ts^{1-\underline{m}}S_{k-2}ds
-\frac{t^{\overline{m}}}{\overline{m}-\underline{m}}\int_0^ts^{1-\overline{m}}S_{k-2}ds.
\end{align*}
For each fixed $i=\overline{m}-1, \cdots, k-2$, the middle summation for $j$ in \eqref{eq-expression-v-k-intermediate}
yields a linear combination of $t^{i+2}(\log t)^j$, for $0\le j\le M_i$. 
By dividing \eqref{eq-expression-v-k-intermediate} by $t^2$, we obtain 
\begin{align}\label{eq-expression-w-k-2-intermediate}
w=\sum_{i=0}^{\overline{m}-3}b_it^i+\sum_{i=\overline{m}-2}^{k-2}\sum_{i=0}^{M_i}b_{i,j}t^{i}(\log t)^j+T_{k-2},
\end{align}
where, for each $i=\overline{m}-1, \cdots, k-2$ and $j=0, \cdots, M_i$, 
$$\text{$b_{i,j}$ is a linear combination of $a_{i,0}, \cdots, a_{i,j}$},$$
and $T_{k-2}$ is given by 
\begin{align}\label{eq-expression-remainder-T-k-2}
T_{k-2}=\frac{t^{\underline{m}-2}}{\overline{m}-\underline{m}}\int_0^ts^{1-\underline{m}}S_{k-2}ds
-\frac{t^{\overline{m}-2}}{\overline{m}-\underline{m}}\int_0^ts^{1-\overline{m}}S_{k-2}ds.
\end{align}
For $i=k-2$ and $j=0, \cdots, M_i$, each $a_{i,j}$ satisfies \eqref{eq-regularity-coefficients-F-k}. 
Hence, each $b_{i,j}$ satisfies \eqref{eq-regularity-b-k}. 
Next, set 
\begin{align*}
\underline{T}_{k-2}
&=\frac{t^{\underline{m}-2}}{\overline{m}-\underline{m}}\int_0^ts^{1-\underline{m}}S_{k-2}ds,\\
\overline{T}_{k-2}&=-\frac{t^{\overline{m}-2}}{\overline{m}-\underline{m}}\int_0^ts^{1-\overline{m}}S_{k-2}ds.
\end{align*}
Then, 
$$T_{k-2}=\underline{T}_{k-2}+\overline{T}_{k-2}.$$
A simple computation yields
\begin{align*}
t\partial_tT_{k-2}= 
(\underline{m}-2) \underline{T}_{k-2}
+(\overline{m}-2)\overline{T}_{k-2},
\end{align*}
and 
\begin{align*}
t^2\partial^2_{t}T_{k-2}= -S_{k-2}
+(\underline{m}-2)(\underline{m}-3) \underline{T}_{k-2}
+(\overline{m}-2)(\overline{m}-3)\overline{T}_{k-2}.
\end{align*}
By \eqref{eq-LinearRegularity-F-m1}, Lemma \ref{lemma-BasicHolderRegularity}, 
and Lemma \ref{lemma-BasicHolderRegularity1}, we have, 
for any $\nu\le k-2$ and $\tau\le \ell-k$, 
\begin{align*}
\partial_t^\nu D^\tau_{y'}\underline{T}_{k-2}\in C^{\epsilon}(\bar G_r),
\quad 
\partial_t^\nu D^\tau_{y'}\overline{T}_{k-2}\in C^{\epsilon}(\bar G_r),\end{align*}
and 
\begin{align*}
|\partial_t^\nu D^\tau_{y'}\underline{T}_{k-2}|
+|\partial_t^\nu D^\tau_{y'}\overline{T}_{k-2}|\le Ct^{k-2-\nu+\alpha}\quad\text{in }G_{r}.\end{align*}
As a consequence, we obtain, 
for any $\nu\le k-2$ and $\tau\le \ell-k$, 
\begin{align*}
\partial_t^\nu D_{y'}^\tau T_{k-2}, 
t\partial_t\partial_t^\nu D_{y'}^\tau T_{k-2}, 
t^2\partial_t^2\partial_t^\nu D_{y'}^\tau T_{k-2}\in C^{\epsilon}(\bar{G}_r), 
\end{align*}
and
\begin{align*}
|\partial_t^\nu D_{y'}^\tau T_{k-2}|
+|t\partial_t\partial_t^\nu D_{y'}^\tau T_{k-2}|
+|t^2\partial_t^2\partial_t^\nu D_{y'}^\tau T_{k-2}|
\le Ct^{k-2-\nu+\alpha}
\quad\text{in }G_{r}.\end{align*}
Therefore, we conclude \eqref{eq-regularity-Rk-equiv} 
for the chosen $k$. This finishes the proof by induction. 
\end{proof}

Next, we prove two corollaries.  We first prove that, if the first logarithmic term does not appear, 
then there are no logarithmic terms in the expansion and the solutions are as regular as the 
nonhomogeneous terms allow. 

\begin{cor}\label{Cor-ExpSmth}
Assume that $\underline{m}$ is a constant and $\overline{m}$ is an integer satisfying 
\eqref{eq-Assumption_m1} and \eqref{eq-Assumption_m2} and 
that $F$ is $C^{\ell-2,\alpha}$ in $y'$ and smooth in other 
arguments,  for some integers $\ell\ge\overline{m}$ and
some $\alpha\in (0,1)$. 
Let $v\in C^{1,1}(\bar G_R)\cap C^{\ell, \alpha}(G_R)$ 
be a solution of \eqref{gtt} in $G_R$, 
for some $R>0$, satisfying 
\eqref{eq-Quasi-TangentialRegu-assumption}  for any $r\in (0, R)$ and any 
$\tau\le \ell-2$. 
If $c_{\overline{m}, 1}=0$, then $c_{i, j}=0$, for any $i=\overline{m}, \cdots, \ell$ and $j=1, \cdots, N_i$, 
and $v\in C^{\ell,\epsilon}(\bar G_r)$, for any $r\in (0, R)$ and $\epsilon\in (0,\alpha)$. 
\end{cor}

\begin{proof} 
First, we have $c_{\overline{m},1}=0$ by the assumption 
and $N_{\overline{m}}=1$ by Lemma \ref{lemma-Main-UpTo-n}. 
Inductively for any $\overline{m}+1\le k\le \ell$, we assume $c_{i,j}=0$ 
for any $\overline{m}\le i\le k-1$ and $1\le j\le N_i$. 
Now, we examine Steps 1-3 in the proof of Theorem \ref{Thm-MainThm} and prove 
$c_{k,j}=0$ for any $1\le j\le N_k$. 
For brevity, we write $a_i=a_{i,0}$ and $b_i=b_{i,0}$, for $\overline{m}-2\le i\le \ell-2$, 
and $c_i=c_{i,0}$, for $\overline{m}\le i\le \ell$. 

By \eqref{eq-relation-c-a-k}, we have $b_{i,j}=0$ for $\overline{m}-2\le i\le k-3$ and $1\le j\le M_i$. 
By \eqref{eq-expression-w-general-k} with $k$ replaced by $k-1$, 
we get 
\begin{align*}\label{eq-expression-w-general-k-no-log}
w=\sum_{i=0}^{k-3}b_i t^i
+T_{k-3}\quad\text{in }G_R.
\end{align*}
Instead of \eqref{eq-identity-F-int-k}, we have 
\begin{equation*}
F=\sum_{i=0}^{k-2} a_it^i
+S_{k-2}\quad\text{in }G_R.\end{equation*}
Then, in computing $v$ in \eqref{eq-expression-v-k-intermediate}, there are no logarithmic terms. In 
other words, we obtain 
\begin{align}\label{eq-expression-v-k-intermediate-no-log}
v=\sum_{i=2}^{k}c_it^i+R_k.
\end{align}
Therefore, $c_{k,j}=0$ for $1\le j\le N_i$. 

By \eqref{eq-expression-v-k-intermediate-no-log}, we conclude $D_{y'}^{\tau} \partial_t^kv\in C^\epsilon(\bar G_r)$, 
for any $\tau\le \ell-k$ and any $\epsilon\in (0,\alpha)$. 
This holds for any $k=\overline{m}, \cdots, \ell$. 
Hence, $D_{y'}^{\tau} \partial_t^\nu v\in C^\epsilon(\bar G_r)$, 
for any nonnegative integers $\tau$ and $\nu$ with $\tau+\nu\le \ell$, 
any $r\in (0, R)$, and any $\epsilon\in (0,\alpha)$.  
This implies the desired result. 
\end{proof}

In the next result, we estimate the largest power of the logarithmic factors. 
For simplicity, we state in the smooth category. 

\begin{cor}\label{cor-EstimateN}
Let $v$, $\underline{m}$, $\overline{m}$, 
and $F$ be as in Theorem \ref{Thm-MainThm}. 
Assume $F$ is smooth in all of its arguments and 
\be\label{eq-Linearity}F\text{ is linear in }\frac{v^2}{t^3},\frac{vv_t}{t^2}, \frac{v_t^2}{t}.\ee
Then, for any $i\ge \overline{m}$, 
\begin{equation}\label{eq-estimate-Ni}N_i\leq \Big[\frac{i-1}{\overline{m}-1}\Big],\end{equation}
where $[x]$ denotes the integer part of $x$. 
\end{cor}

\begin{proof} We will prove the following statement: 
If $t^k(\log t)^j$ appears in $v$ for some $k\ge \overline{m}$, then 
\be\label{eq-InductionN}k\ge (\overline{m}-1)j+1.\ee
We prove  \eqref{eq-InductionN} by an induction on $k$. If $k=\overline{m}$, then the corresponding 
$j$ is either 0 or 1, and hence  \eqref{eq-InductionN} holds. 
Suppose \eqref{eq-InductionN} holds for any $k=\overline{m}, \cdots, l-1$ 
for some $l\ge \overline{m}+1$. We now consider $k=l$. The proof is by a 
computation based on \eqref{gtt} and \eqref{gtt-Coefficients}. 

Let $t^l(\log t)^j$ be one term in $v$. 
We substitute such a term in \eqref{gtt} and note 
\begin{align*}
&\big(t^l(\log t)^j\big)^{\prime\prime}+p\frac{\big(t^l(\log t)^j\big)^{\prime}}{t}
+q\frac{t^l(\log t)^j}{t^2}\\
&\qquad=(l-\overline{m})(l-\underline{m})t^{l-2}(\log t)^j+j(2l-1+p)t^{l-2}(\log t)^{j-1}\\
&\qquad\qquad+j(j-1)t^{l-2}(\log t)^{j-2}.\end{align*}
The term $t^{l-2}(\log t)^j$ has the highest power of $\log t$ and 
a nonzero coefficient since $l\ge \overline{m}+1$. 
Next, we find the corresponding term in $F$. 
Set
\begin{align*}
\overline{v}(y',t)&=\sum_{i=2}^{l-1} \sum_{j=0}^{N_i} c_{i, j}(y')t^{i}(\log t)^j. 
\end{align*}
By the induction hypothesis, each pair $i$ and $j$ in $\overline{v}$ satisfy 
\eqref{eq-InductionN}, with $k$ replaced by $i$. We now substitute $\overline{v}$ in \eqref{gtt}
and identify $t^{l-2}(\log t)^j$. 
First, we note that 
all terms $t^i(\log t)^j$ in $\overline{v}_t, \frac{\overline{v}}{t}$ satisfy
$i\ge (\overline{m}-1)j$ and that all terms $t^i(\log t)^j$ in 
$\frac{\overline{v}_t}{t}, \frac{\overline{v}}{t^2}$ satisfy $i\ge (\overline{m}-1)j-1$. Hence, 
all terms $t^i(\log t)^j$ in 
$\frac{\overline{v}^2}{t^3},\frac{\overline{v} \overline{v}_t}{t^2}, 
\frac{\overline{v}_t^2}{t}$ with $i\le l-2$
satisfy $i\ge (\overline{m}-1)j-1$. Therefore, these three terms are dominant. 
Recall that $F$ is smooth in 
$$\overline{v}_t,\, \frac{\overline{v}}{t},\, \frac{\overline{v}^2}{t^3},\,
\frac{\overline{v} \overline{v}_t}{t^2},\,
\frac{\overline{v}_t^2}{t},$$ 
and is linear in the last three quantities by \eqref{eq-Linearity}. 
If we expand $F$ in terms of $t^i(\log t)^j$, then all terms $t^i(\log t)^j$ with $i\le l-2$ 
satisfy $i\ge (\overline{m}-1)j-1$. With $i=l-2$, we obtain 
$l\ge (\overline{m}-1)j+1$. 
This is \eqref{eq-InductionN} for $k=l$. 
\end{proof}

We now make two remarks to end this section. 

\begin{remark}\label{remark-double-indices}
Theorem \ref{Thm-MainThm} is formulated with two indices $\ell$ and $k$, 
with $\ell$ for the regularity and $k$ for the order of the expansion 
in \eqref{eq-expression_g_n2z}. In this formulation, the remainder $R_k$ in \eqref{eq-expression_g_n2z}
has an optimal regularity. Had we expanded $v$ up to degree $\ell$, 
the corresponding $R_\ell$ would have a much lower regularity since it contains derivatives of $v$ up to order $\ell$. 
The present formulation of Theorem \ref{Thm-MainThm} plays an important role in Section \ref{Sec-Decomposition}, 
in particular, in the proof of Theorem \ref{thrm-Decomposition-finite}. 
\end{remark}

\begin{remark}\label{remark-integer}
Results in this section 
are established based on an essential assumption that $\overline{m}$ is an integer. 
Under this assumption, the function $\log t$ and its powers provide obstacles to the higher regularity. 
If $\overline{m}$ is not an integer, we introduce $\gamma=\overline{m}-[\overline{m}]\in (0,1)$, 
where $[\overline{m}]$ is the integer part of $\overline{m}$. 
In this case, $t^\gamma$ and its powers provide obstacles to the higher regularity. 
If $\gamma$ is irrational, there are infinitely many such powers. 
If $\gamma$ is rational, there are only finitely many such powers. 
However, it is quite complicated to formulate a result similar as Theorem \ref{Thm-MainThm}. 
The optimal regularity has different forms for the three different cases $\alpha<\gamma$, $\alpha=\gamma$, and 
$\alpha>\gamma$. We will not pursue along this direction. 
When we write our main equation \eqref{eq-Intro-Equ} in the form \eqref{gtt}, 
we have $\overline{m}=n+1$, an integer. Refer to Section \ref{sec-Setup} for details. 
\end{remark} 
 
\section{The Concise Boundary Regularity}\label{Sec-Decomposition}

In this section, we prove two results concerning concise boundary regularity. 

Let $v$ be the solution as in Theorem \ref{Thm-MainThm}. 
We intend to decompose $v$ according to the singular form $\log t$ and analyze the regularity of the rest of functions. 
Under the assumptions of Theorem \ref{Thm-MainThm}, there exist 
$c_i\in C^{\ell-i,\alpha}(B'_R)$, for $2\le i\le \overline{m}-1$,  
and $c_{i,j}\in C^{\ell-i,\epsilon}(B'_R)$, for $\overline{m}\le i\le \ell$, $0\le j\le N_i$, and any $\epsilon\in (0,\alpha)$, such that 
\begin{align*}
v=\sum_{i=2}^{\overline{m}-1}c_i t^i
+\sum_{i=\overline{m}}^{\ell}\Big[ \sum_{j=0}^{N_i} c_{i, j}(\log t)^j\Big]t^{i}
+R_{\ell},\end{align*}
where $R_\ell$ is the remainder. 
For each $i\ge \overline{m}$,  $N_i$ is the largest power of $\log t$ for $t^i$. 
We now introduce a sequence of positive integers $\{I_j\}$ such that 
$$j\le N_i\text{ if and only if }i\ge I_j,$$
with the equivalence of equalities. Then, $I_j\ge\overline{m}$ for any $j\ge 1$. 
We write $v$ as 
\begin{align}\label{eq-expansion-ell}
v=\sum_{i=2}^{\ell} c_{i, 0}t^{i}
+\sum_{j=1}^{N_\ell}\Big[ \sum_{i=I_j}^{\ell} c_{i, j}t^{i}\Big](\log t)^j
+R_{\ell},\end{align}
where $c_{i,0}=c_i$ for $i=2, \cdots, \overline{m}-1$. 
Then, 
\begin{align}\label{eq-expansion-ell-log}
v=w_0+w_1\log t+\cdots+w_{N_\ell}(\log t)^{N_\ell},
\end{align}
where 
\begin{align*}w_0&=\sum_{i=2}^{\ell} c_{i, 0}(y')t^{i}+R_{\ell}(y',t),\\
w_j&=\sum_{i=I_j}^{\ell} c_{i, j}(y')t^{i}\quad\text{for }j=1, \cdots, N_{\ell}.
\end{align*}
By the regularity of $c_{i,j}$ in Theorem \ref{Thm-MainThm}, we know 
$w_0, w_1, \cdots, w_{N_\ell}\in C^{\epsilon}(\bar{G}_r)$, for any $r\in (0, R)$. 
Comparing with the assumption that $F$ is $C^{\ell-2,\alpha}$ in all its argument, there is a huge loss of regularity. 
In one main result in this section, we will establish \eqref{eq-expansion-ell-log}
for appropriate $w_0, w_1, \cdots, w_{N_\ell}\in C^{\ell, \epsilon}(\bar{G}_r)$, for any $r\in (0, R)$.
We will use the full power of Theorem \ref{Thm-MainThm} 
for this purpose. Refer to Remark \ref{remark-double-indices}. 

First, we need an extension lemma. 

\begin{lemma}\label{lemma-extensions} Let $\ell$ be a nonnegative integer, $\alpha\in (0,1)$ 
be a constant, and $c_0, c_1, \cdots, c_\ell$ be given functions on $\bar B'_R$ with 
$c_i\in C^{\ell-i,\alpha}(\bar B'_R)$ for $0\le i\le \ell$. Then, there exists 
a function $w\in C^{\ell,\alpha}(\bar G_R)$ such that, 
for any $0\le i\le \ell$,  
\begin{equation}\label{eq-Taylor-degree}
\partial^i_tw(\cdot, 0)=c_{i}\quad\text{on }B_R'.\end{equation}
\end{lemma}

By \eqref{eq-Taylor-degree}, the Taylor polynomial of $w$ with respect to $t$ 
of degree $\ell$ is given by
$$\sum_{i=0}^{\ell}\frac{1}{i!}c_{i}(y') t^i.$$
However, we cannot define $w$ simply by the above sum since it does not have 
a sufficient regularity. 
With the help of the theory of elliptic PDEs, we can construct $w$ by solving 
an elliptic equation of a sufficiently high order. 

\begin{proof}[Proof of Lemma \ref{lemma-extensions}] 
Consider a domain $\Omega$ with a smooth boundary 
such that $G_R\subseteq \Omega$ and $\partial G_R\cap \{t=0\}\subseteq\partial\Omega$, 
and extend $c_i$ to $\partial\Omega$ such that $c_i\in C^{\ell-i, \alpha}(\partial\Omega)$ for $0\le i\le \ell$. 
We solve 
\begin{align*}\Delta^{2(\ell+1)}w&=0\quad\text{in }\Omega,\\
\partial_\nu^iw&=c_i\quad\text{on }\partial\Omega,\,\text{ for }0\le i\le \ell,
\end{align*}
where $\nu$ is the interior normal vector to $\partial\Omega$. 
Then, such a $w$ has the required regularity. 
\end{proof}

In Corollary \ref{cor-EstimateN}, we proved \eqref{eq-estimate-Ni} under the assumption \eqref{eq-Linearity}. 
Note that 
\begin{equation}\label{eq-equivalence-i-j}j\le \frac{i-1}{\overline{m}-1}\, \Longleftrightarrow\,  i\ge j(\overline{m}-1)+1.
\end{equation}
Hence, \eqref{eq-estimate-Ni} implies 
\begin{equation}\label{eq-assumption-I}
I_j\ge j(\overline{m}-1)+1.
\end{equation}
In particular, $I_j\ge j$ since $\overline{m}\ge 2$.

We first discuss the case that $F$ is finitely differentiable with respect to $y'$. 

\begin{theorem}\label{thrm-Decomposition-finite}
Assume that $\underline{m}$ is a constant and $\overline{m}$ is an integer satisfying 
\eqref{eq-Assumption_m1} and \eqref{eq-Assumption_m2} and 
that $F$ is $C^{\ell-2,\alpha}$ in $y'$ and smooth in other 
arguments,  for some integers $\ell\ge\overline{m}$ and
some $\alpha\in (0,1)$, and satisfies \eqref{eq-Linearity}. 
Let $v\in C^{1,1}(\bar G_R)\cap C^{\ell, \alpha}(G_R)$ 
be a solution of \eqref{gtt} in $G_R$, 
for some $R>0$, satisfying
\eqref{eq-Quasi-TangentialRegu-assumption}  for any $r\in (0, R)$ and any 
$\tau\le \ell-2$. 
Then, there exist functions  $w_0, w_1, \cdots, w_{N_\ell}\in C^{\ell, \epsilon}(\bar{G}_r)$,  
for any $\epsilon\in (0,\alpha)$ and any $r\in (0,R)$, 
such that
\begin{equation}\label{eq-Main-decomposition-finite}
v=\sum_{j=0}^{N_\ell} w_j(\log t)^j\quad\text{in }G_R,\end{equation}
and, for each $j=1, \cdots, N_\ell$, 
\begin{equation}\label{eq-Main-vanishing-finite}
\partial_{t}^iw_j\big|_{t=0}=0\quad\text{for any }i=0, 1, \cdots, I_j-1.
\end{equation}
If, in addition, $\partial_t^{I_1}w_1(\cdot, 0)=0$ on $B_R'$, then $w_1, \cdots, w_{N_\ell}$ 
are identically zero and $v\in C^{\ell, \epsilon}(\bar{G}_r)$,  
for any $\epsilon\in (0,\alpha)$ and any $r\in (0,R)$. 
\end{theorem} 

If we expand $w_j$ in terms of $t$, \eqref{eq-Main-vanishing-finite} implies that 
$w_j$ starts with the term $t^{I_j}$. 

\begin{proof} By Theorem \ref{Thm-MainThm}, there exist functions
$c_{i,0}\in C^{\ell-i,\alpha}(B'_R)$, for $2\le i\le\overline{m}-1$,  
and $c_{i,j}\in C^{\ell-i,\epsilon}(B'_R)$, for $\overline{m}\le i\le \ell$, $0\le j\le N_i$, and any $\epsilon\in (0,\alpha)$, 
such that \eqref{eq-expansion-ell} holds, i.e., 
$$v=\sum_{i=2}^{\ell} c_{i, 0}t^{i}
+\sum_{j=1}^{N_\ell} \sum_{i=I_j}^{\ell} c_{i, j}t^{i}(\log t)^j
+R_{\ell}.$$
We set $c_{i,j}=0$ 
when the corresponding term is absent. 
For each $j=1, \cdots, N_\ell$, by Lemma \ref{lemma-extensions}, 
there exists a function $w_j\in C^{\ell, \epsilon}(\bar{G}_r)$,  
for any $r\in (0, R)$ and any $\epsilon\in (0,\alpha)$, such that  
\begin{equation}\label{eq-IV-wj}\partial_{t}^iw_j\big|_{t=0}=
\begin{cases} 0&\text{for }i=0, \cdots, I_j-1,\\
i!c_{i,j}&\text{for  }i=I_j, \cdots, \ell.\end{cases}\end{equation}
Set 
\begin{equation}\label{eq-definition-w0}w_0=v-\sum_{j=1}^{N_\ell} w_j(\log t)^j.\end{equation}
We will prove $w_0\in C^{\ell, \epsilon}(\bar{G}_r)$,  
for any $r\in (0, R)$ and $\epsilon\in (0,\alpha)$.

Fix an integer $k$ with $\overline{m}\le k\le\ell$ 
and fix an $r\in (0, R)$. We will prove that $\partial_t^\nu D_{y'}^\tau w_0\in C^\epsilon(\bar{G}_r)$, 
for any $\nu\le k$ and $\tau\le \ell-k$. 
By Theorem \ref{Thm-MainThm}, we have 
\begin{equation}\label{eq-definition-v-k}
v=\sum_{i=2}^{k} c_{i, 0}t^{i}
+\sum_{j=1}^{N_k} \sum_{i=I_j}^{k} c_{i, j}t^{i}(\log t)^j
+R_{k},\end{equation}
where $R_k$ is a function in $G_r$ such that
$D_{y'}^\tau\partial_t^\nu R_k\in C^{\epsilon}(\bar{G}_r)$, for 
any $\nu\le k$,  $\tau\le \ell-k$, and any $\epsilon\in (0,\alpha)$.
For each $j=1, \cdots, N_k$, by \eqref{eq-IV-wj}, we can write 
\begin{equation}\label{eq-definition-wj-k}w_j=\sum_{i=I_j}^kc_{i,j}t^i+S_{j,k},\end{equation}
where $S_{j,k}$ is a function in $G_r$ such that, for 
any $\nu\le k$,  $\tau\le \ell-k$, and any $\epsilon\in (0,\alpha)$, 
\begin{equation}\label{eq-property-wj-k1}
D_{y'}^\tau\partial_t^\nu S_{j,k}\in C^{k-\nu, \epsilon}(\bar{G}_r),\end{equation}
and 
\begin{equation}\label{eq-property-wj-k2}|D_{y'}^\tau\partial_t^\nu S_{j,k}|\le Ct^{k-\nu+\alpha}.\end{equation}
By \eqref{eq-definition-w0}, \eqref{eq-definition-v-k}, and \eqref{eq-definition-wj-k}, we get 
$$w_0=\sum_{i=2}^{k} c_{i, 0}t^{i}+R_k
-\sum_{j=1}^{N_k}S_{j,k}(\log t)^j-\sum_{j=N_{k}+1}^{N_\ell} w_j(\log t)^j.$$
Take any $\nu\le k$ and $\tau\le \ell-k$. 
By Theorem \ref{Thm-MainThm}, we have 
$\partial_t^\nu D_{y'}^\tau (c_{i, 0}t^{i})\in C^\epsilon(\bar{G}_r)$  for $2\le i\le k$, 
and $\partial_t^\nu D_{y'}^\tau R_k\in C^\epsilon(\bar{G}_r)$. 
For $1\le j\le N_k$, by \eqref{eq-property-wj-k1} and \eqref{eq-property-wj-k2}, we have 
$\partial_t^\nu D_{y'}^\tau \big(S_{j,k}(\log t)^j\big)\in C^\epsilon(\bar{G}_r)$. 
For $N_{k}+1\le j\le N_\ell$, by \eqref{eq-assumption-I} and \eqref{eq-Main-vanishing-finite}, 
we have $I_j\ge k+1$ and hence $\partial_t^\nu D_{y'}^\tau \big(w_j(\log t)^j\big)\in C^\epsilon(\bar{G}_r)$. 
In summary, we obtain $\partial_t^\nu D_{y'}^\tau w_0\in C^\epsilon(\bar{G}_r)$. 
This holds for any $\nu\le k$ and $\tau\le \ell-k$, and hence for any $\nu$ and $\tau$ with $\nu+\tau\le \ell$.
Therefore, $w_0\in C^{\ell, \epsilon}(\bar{G}_r)$,  
for any $\epsilon\in (0,\alpha)$ and any $r\in (0,R)$.
\end{proof} 

Next, we discuss the case that $F$ is smooth in all its arguments. 
As $\ell$ increases, the number of terms in the right-hand side of \eqref{eq-Main-decomposition-finite} 
also increases. It is not clear that the summation in \eqref{eq-Main-decomposition-finite} converges 
as $\ell\to\infty$. In the next result, we treat $t^{\overline{m}-1}\log t$ as an additional self-variable.

\begin{theorem}\label{thrm-Decomposition-infinite}
Assume that $\underline{m}$ is a constant and $\overline{m}$ is an integer satisfying 
\eqref{eq-Assumption_m1} and \eqref{eq-Assumption_m2} and 
that $F$ is smooth in its arguments and satisfies \eqref{eq-Linearity}. 
Let $v\in C^{1,1}(\bar G_R)\cap C^{\ell, \alpha}(G_R)$ 
be a solution of \eqref{gtt} in $G_R$, 
for some $R>0$, satisfying  
\eqref{eq-Quasi-TangentialRegu-assumption} 
for any $r\in (0, R)$ and for any $\tau\ge 0$. 
Then, $v$ is smooth in $y', t$, and $t^{\overline{m}-1}\log t$ in $\bar{G}_r$,  
for  any $r\in (0,R)$. 
If, in addition, $\partial_{t}\partial_{t^{\overline{m}-1}\log t}v|_{t=0}=0$ on $B_R'$, 
then $v\in C^{\infty}(\bar{G}_r)$,  
for any $r\in (0,R)$. 
\end{theorem}

\begin{proof} By applying Theorem \ref{Thm-MainThm} successively for each $\ell\ge \overline{m}$, we obtain 
smooth functions $c_2, \cdots, c_{\overline{m}-1}$ and $c_{i,j}$ for $i\ge \overline{m}$ and $0\le j\le N_i$ in $B_R'$ such that 
\begin{align}\label{eq-expression-any-ell}
v(y', t)=\sum_{i=2}^{\overline{m}-1}c_i(y') t^i
+\sum_{i=\overline{m}}^{\ell} \sum_{j=0}^{N_i} c_{i, j}(y')t^{i}(\log t)^j
+R_{\ell}(y',t),\end{align}
where $R_\ell$ is $C^{\ell,\epsilon}$ in $\bar{G}_r$, for any $\epsilon\in (0,\alpha)$ and any $r\in (0,R)$. 
For unification, we write $c_{i,0}=c_{i}$ for $2\le i\le \overline{m}-1$. Consider a formal expansion 
\begin{align*}
v_*=\sum_{i=1}^\infty \sum_{j=0}^{N_i}c_{i,j}t^i(\log t)^j, 
\end{align*}
where we set $c_{i,j}=0$ when such term is absent in the expansion of $v$. 
Under the assumption \eqref{eq-Linearity}, \eqref{eq-estimate-Ni} holds. 
Note the equivalence in \eqref{eq-equivalence-i-j} and the estimate \eqref{eq-assumption-I}. 
By exchanging the order of summations, we get 
\begin{align*}
v_*=\sum_{j=0}^\infty \sum_{i=j(\overline{m}-1)+1}^{\infty}c_{i,j}t^i(\log t)^j.
\end{align*}
For each fixed pair $i$ and  $j$, introduce $k$ such that 
\begin{align}\label{eq-relation-ijk} 
i=j(\overline{m}-1)+k.\end{align} 
Throughout the proof, any $i,j,k$ are related by \eqref{eq-relation-ijk}. 
Then, 
\begin{align*}
v_*=\sum_{j=0}^\infty \sum_{k=1}^{\infty}c_{j(\overline{m}-1)+k,j}t^{j(\overline{m}-1)+k}(\log t)^j
=\sum_{j=0}^\infty \sum_{k=1}^{\infty}c_{j(\overline{m}-1)+k,j}t^k(t^{\overline{m}-1}\log t)^j.
\end{align*}
We now write 
\begin{align}\label{eq-Form-Series}
v_*=\sum_{l=1}^\infty\sum_{k+j=l}c_{j(\overline{m}-1)+k,j}t^k(t^{\overline{m}-1}\log t)^j.
\end{align}
We emphasize that \eqref{eq-Form-Series} is a formal expansion. The series in the right-hand side may not converge. 
Following \cite{HanWang}, we now introduce a sequence of cutoff functions to construct a convergent series. 

Set 
\begin{align}\label{eq-definition-TS} 
T =t,\quad S =-t^{\overline{m}-1}\log t.\end{align} 
For simplicity, we assume $R\in (0,1)$. For each $l\ge 1$, set 
\begin{align}\label{eq-definition-w-l} 
w_l(y',T,S)=\sum_{k+j=l}(-1)^jc_{j(\overline{m}-1)+k,j}(y')T^kS^j.
\end{align}
This is a homogeneous polynomial in $T$ and $S$ of degree $l$, with coefficients given by smooth functions of $y'$. 
Take a cutoff function $\eta\in C_0^\infty(-R,R)$ with $\eta=1$ on $(-R/2,R/2)$. 
For any constant $\lambda_l\ge 1$, we observe that
$\eta(\lambda_l {T}) =0$ if $|{T} |\geq {1}/{\lambda_l}$ and $\eta(\lambda_l {T}) =1$ if  $|{T} |\leq {1}/{(2\lambda_l)}$. 
A similar result holds for $\eta(\lambda_l {S})$. 
It is straightforward to verify that, for any $\alpha \in \mathbb{Z}_+^{n+1}$ with $|\alpha|<l$, 
\begin{align*}
\big| \partial^\alpha_{y',T,S} \big[w_l(y', T, S) \eta(\lambda_l{T})\eta(\lambda_l{S})\big] \big| 
\leq \frac{C_{l, \alpha}}{\lambda_l^{l - |\alpha|}},
\end{align*}
for some positive constant $C_{l, \alpha}$ depending only on $w_l$ and $\alpha$. 
By choosing $\lambda_l$ sufficiently large, we get 
\begin{align}\label{eq-decay-w}
\|w_l(y', T, S) \eta(\lambda_l{T})\eta(\lambda_l{S})\|_{C^{l-1}(B_R\times I\times I)} 
\leq 2^{-l},
\end{align}
where $I=(-R, R)$. Next, set 
$$w(y',T,S)=\sum_{l=1}^\infty w_l(y', T, S)\eta(\lambda_l{T})\eta(\lambda_l{S}).$$
We note that each term in the summation is smooth in $y', T$, and $S$. 
By \eqref{eq-decay-w}, we obtain, for any positive integer $m$, 
\begin{align*}
\sum_{l=m+1}^\infty\|w_l(y', T, S) \eta(\lambda_l{T})\eta(\lambda_l{S})\|_{C^{m}(B_R\times I\times I)} 
\leq \sum_{l=m+1}^\infty 2^{-l}=2^{-m}.
\end{align*}
Hence, $w$ is $C^m$ in $y', T$, and $S$. This is true for any positive integer $m$. 
Therefore, $w$ is smooth in $y', T$, and $S$.

Next, with \eqref{eq-definition-TS}, we have 
\begin{align}\label{eq-expression-w}
w(y',t,-t^{\overline{m}-1}\log t)=\sum_{l=1}^\infty w_l(y', t, -t^{\overline{m}-1}\log t)
\eta(\lambda_l{t})\eta(-\lambda_l{t^{\overline{m}-1}\log t}).\end{align}
For each $l\ge 1$, we set 
\begin{equation}\label{eq-definition-w-tilde-l1}\widetilde w_l(y',t)=w_l(y', t, -t^{\overline{m}-1}\log t).\end{equation}
By \eqref{eq-definition-w-l}, we have 
\begin{equation}\label{eq-definition-w-tilde-l2}
\widetilde w_l(y',t)=\sum_{k+j=l}c_{j(\overline{m}-1)+k,j}t^{j(\overline{m}-1)+k}(\log t)^j.
\end{equation}
With $k+j=l$, we have $j(\overline{m}-1)+k=l+j(\overline{m}-2)$, and hence
$$\widetilde w_l(y',t)=\sum_{k+j=l}c_{l+j(\overline{m}-2),j}t^{l+j(\overline{m}-2)}(\log t)^j.$$
We always assume $\overline{m}\ge 2$. Then, the lowest power of $t$ is $l$, given by $j=0$. 
Therefore, by taking $\lambda_l$ larger if necessary, we obtain 
\begin{align}\label{eq-decay-w-tilde}
\|\widetilde w_l(y',t)\eta(\lambda_l{t})\eta(-\lambda_l t^{\overline{m}-1}\log t)\|_{C^{l-1}(B_R\times I)} 
\leq 2^{-l}.
\end{align}
Fix an arbitrary $m\ge \overline{m}$ and set 
\begin{align}\label{eq-definition-Sm} 
S_m(y',t)=\sum_{l=m+1}^\infty \widetilde w_l(y',t)\eta(\lambda_l{t})\eta(-\lambda_l t^{\overline{m}-1}\log t).
\end{align} 
By \eqref{eq-decay-w-tilde}, we obtain
\begin{align*}
\sum_{l=m+1}^\infty\|\widetilde w_l(y',t)\eta(\lambda_l{t})\eta(-\lambda_l t^{\overline{m}-1}\log t)\|_{C^{m}(B_R\times I)}  
\leq \sum_{l=m+1}^\infty 2^{-l}=2^{-m}.
\end{align*}
Hence, $S_m\in C^m(\bar B_r\times I)$ for any $r\in (0,R)$. 
By \eqref{eq-expression-w}, \eqref{eq-definition-w-tilde-l1}, and \eqref{eq-definition-Sm}, we obtain 
\begin{align}\label{eq-decomposition-w-9}
w(y',t,-t^{\overline{m}-1}\log t)=\sum_{l=1}^m \widetilde w_l(y', t)
\eta(\lambda_l{t})\eta(-\lambda_l{t^{\overline{m}-1}\log t})+S_m(y't).\end{align}
We point out that summation in the right-hand side contains terms with regularity worse than $C^m$. 
We now prove that this summation also contains all terms in $v$ with regularity worse than $C^m$. 

For the same arbitrarily fixed $m\ge \overline{m}$, set 
\begin{equation}\label{eq-definition-M}
M=\max\{j(\overline{m}-1)+k;\, k+j\le m, k\ge 1, j\ge 0\}.\end{equation}
It is easy to see that $M\ge m$. By Theorem \ref{Thm-MainThm}, we have 
\begin{align}\label{eq-expression-v-M}
v=\sum_{i=1}^{M} \sum_{j=0}^{N_i} c_{i, j}t^{i}(\log t)^j
+R_{M},\end{align}
and $R_M\in C^M(\bar{G}_r)\subset C^m(\bar{G}_r)$. 
By the definition of $M$ in \eqref{eq-definition-M}, any $i$ as in \eqref{eq-relation-ijk} with $k+j\le m$ satisfies $i\le M$. 
Hence, we can decompose the summation in the right-hand side of \eqref{eq-expression-v-M} into two parts
according to $k+j\le m$ or $k+j\ge m+1$. Thus, 
\begin{align*}v&=\sum_{k+j\le m}c_{j(\overline{m}-1)+k,j}t^{j(\overline{m}-1)+k}(\log t)^j\\
&\qquad+\sum_{\substack{k+j\ge m+1\\ i\le M, j\le N_i}}c_{j(\overline{m}-1)+k,j}t^{j(\overline{m}-1)+k}(\log t)^j+R_M.\end{align*}
In the second summation, the power of $t$ is given by  $j(\overline{m}-1)+k\ge j+k\ge m+1$. 
We now use the elementary fact that $t^{m+1}(\log t)^j$ is a $C^m$ function on $[0,1]$. Hence, 
each term in the second summation is in $C^m(\bar{G}_r)$, for any $r\in (0,R)$. 
We can combine the second summation with $R_M$ and write 
\begin{align*}v=\sum_{k+j\le m}c_{j(\overline{m}-1)+k,j}t^{j(\overline{m}-1)+k}(\log t)^j+\widetilde{R}_m,\end{align*}
for some $\widetilde{R}_m\in C^m(\bar{G}_r)$, for any $r\in (0,R)$. 
By \eqref{eq-definition-w-l}, we obtain 
\begin{align*}v=\sum_{l=1}^{m}\widetilde w_l+\widetilde{R}_m.\end{align*}
By introducing the cutoff functions $\eta(\lambda_l{t})\eta(-\lambda_l t^{\overline{m}-1}\log t)$ for each term in the summation, 
we write
\begin{align*}v=\sum_{l=1}^{m}\widetilde w_l\eta(\lambda_l{t})\eta(-\lambda_l t^{\overline{m}-1}\log t)
+\sum_{l=1}^{m}\widetilde w_l\big[1-\eta(\lambda_l{t})\eta(-\lambda_l t^{\overline{m}-1}\log t)\big]
+\widetilde{R}_m.\end{align*}
In the second summation, each term is 0 for small $t$ and hence is smooth in $\bar{G}_r$ for any $r\in (0,R)$. 
Therefore, 
\begin{align}\label{eq-definition-w-l2}v(y',t)=\sum_{l=1}^{m}
\widetilde w_l(y',t)\eta(\lambda_l{t})\eta(-\lambda_l t^{\overline{m}-1}\log t)
+\widehat{R}_m(y',t),\end{align}
for some $\widehat{R}_m\in C^m(\bar{G}_r)$ for any $r\in (0,R)$.

Finally, we set 
$$u(y', t)= v(y', t)-w(y', t, t^{\overline{m}-1}\log t).$$
By \eqref{eq-definition-w-l2} and \eqref{eq-decomposition-w-9}, we obtain 
$u=\widehat{R}_m-S_m\in C^m(\bar{G}_r)$ for any $r\in (0,R)$. This holds for any $m\ge\overline{m}$. 
Thus, $u\in C^\infty(\bar{G}_r)$ for any $r\in (0,R)$. 
As a consequence,  
$v = u(y', T) + w(y', T, S)$ is smooth in $y', T, S$ and the theorem is proved. 
\end{proof} 

Without the assumption \eqref{eq-Linearity}, we can prove a weaker result
that $v$ is smooth in $y', t$, and $t\log t$ in $\bar{G}_r$,  
for  any $r\in (0,R)$. 

\appendix 

\section{Calculus Lemmas}\label{Appen-CalculusL}

In this section, we list several calculus lemmas concerning 
H\"older continuous functions. 
We denote by $y=(y', t)$ points in $\mathbb R^n$ and set, for any constant $r>0$, 
$$G_r=B_r'\times (0,r).$$ 

\begin{lemma}\label{lemma-BasicHolderRegularity}  Let $k$ and $l$ be nonnegative integers, 
$a> 0$ and $\alpha\in (0,1]$ be constants, and $f\in C^\alpha(\bar G_r)$ be a function. 
Define, 
for any $(y',t)\in G_r$, 
$$F(y', t)=\frac{1}{t^{a}}\int_0^ts^{a-1}f(y',s)ds.$$
Suppose 
$\partial_t^\nu D_{y'}^\tau f \in C^{\alpha}(\bar{G}_r)$, for any 
$\tau\le k$ and $\nu\le l$. 
Then, $\partial_t^\nu D_{y'}^\tau F\in C^{\alpha}(\bar{G}_r)$, for any 
$\tau\le k$ and $\nu\le l$. 
Moreover, if $\partial_t^\nu f(y',0)=0$ for any $\nu\le l$, then, for any $\tau\le k$ and 
$\nu\le l$ and any $(y',t)\in G_r$,  
$$|\partial_t^\nu D_{y'}^\tau F(x',t)|\le C[\partial_t^l D_{y'}^\tau f]_{C^\alpha(\bar G_r)}t^{l-\nu+\alpha},$$ 
where $C$ is a positive constant depending only on $l$, $a$ and $\alpha$. 
\end{lemma}

\begin{lemma}\label{lemma-BasicHolderRegularity2}  
Let $k$ and $l$ be nonnegative integers, 
$a$ and $\alpha\in (0,1]$ be constants with $a>l+\alpha$, and $f\in C^\alpha(\bar G_r)$ be a function. 
Define, 
for any $(y',t)\in G_r$, 
$$F(y', t)=t^{a}\int_t^rs^{-a-1}f(y',s)ds.$$
Suppose 
$\partial_t^\nu D_{y'}^\tau f \in C^{\alpha}(\bar{G}_r)$, for any 
$\tau\le k$ and $\nu\le l$. 
Then, $\partial_t^\nu D_{y'}^\tau F\in C^{\alpha}(\bar{G}_r)$, for any 
$\tau\le k$ and $\nu\le l$. 
Moreover, if $\partial_t^\nu f(y',0)=0$ for any $\nu\le l$, then, for any $\tau\le k$ and 
$\nu\le l$ and any $(y',t)\in G_r$,  
$$|\partial_t^\nu D_{y'}^\tau F(y',t)|\le C[\partial_t^l D_{y'}^\tau f]_{C^\alpha(\bar G_r)}t^{l-\nu+\alpha},$$ 
where $C$ is a positive constant depending only on $l$, $a$ and $\alpha$. 
\end{lemma}

We now generalize Lemma \ref{lemma-BasicHolderRegularity} 
to allow $a$ to be a nonpositive constant. 
We rename the constant $a$ and consider only the case of a  
nonpositive integer. 

\begin{lemma}\label{lemma-BasicHolderRegularity1}  
Let $k, l, m$, and $a$ be nonnegative integers with $l\ge m\ge a$, 
$\alpha\in (0,1]$ be a constant, and $f\in C^\alpha(\bar G_r)$ be a function. 
Define, 
for any $(y',t)\in G_r$, 
$$F(y', t)=t^{a}\int_0^ts^{-a-1}f(y',s)ds.$$
Suppose 
$\partial_t^\nu D_{y'}^\tau f \in C^{\alpha}(\bar{G}_r)$ for any 
$\tau\le k$ and $\nu\le l$, and 
$\partial_t^\nu D_{y'}^\tau f(\cdot, 0)=0$ for any 
$\tau\le k$ and $\nu\le m$. 
Then, for any 
$\tau\le k$ and $\nu\le l$ and any $\epsilon\in (0,\alpha)$, 
$$\partial_t^\nu D_{y'}^\tau F\in C^{\epsilon}(\bar{G}_r),$$ 
and, for any 
$\tau\le k$ and $\nu\le m$ and any $(y',t)\in G_r$,  
$$|\partial_t^\nu D_{y'}^\tau F(y',t)|
\le C[\partial_t^m D_{y'}^\tau f]_{C^\alpha(\bar G_r)}t^{m-\nu+\alpha},$$ 
where $C$ is a positive constant depending only on $l$ and $\alpha$. 
\end{lemma}

We note that there is a slight loss of regularity for $F$ in Lemma \ref{lemma-BasicHolderRegularity1}
in the case $l=m=a$. 

We also need the following result on the regularity of functions of $y'$. 

\begin{lemma}\label{lemma-BasicHolderRegularity1-int0}  
Let $a$ and $k$ be nonnegative integers, 
$\alpha\in (0,1)$ be a constant,  
and $f\in C^\alpha(\bar G_r)$ be a function. 
Define, 
for any $y'\in B'_r$, 
$$b(y')=\int_0^rs^{-a-1}f(y',s)ds.$$
Suppose 
$\partial_t^\nu D_{y'}^\tau f \in C^{\alpha}(\bar{G}_r)$ and 
$\partial_t^\nu D_{y'}^\tau f(\cdot, 0)=0$, for any 
$\tau\le k$ and $\nu\le a$. 
Then, for any $\epsilon\in (0,\alpha)$,  
$$b\in C^{k,\epsilon}(\bar B'_r).$$ 
\end{lemma}

For an illustration, we present a proof of Lemma \ref{lemma-BasicHolderRegularity1}, 
demonstrating the loss of regularity.

\begin{proof}[Proof of Lemma \ref{lemma-BasicHolderRegularity1}] 
We only consider $k=0$. 

First, we derive estimates of $F$. By a change of 
variables $s=t\rho$, we have 
$$F(y',t)=\int_0^1\rho^{-a-1}f(y',t\rho)d\rho,$$
and hence, for any $0\le \nu\le l$, 
$$\partial_t^\nu F(y',t)=\int_0^1\rho^{\nu-a-1}\partial_t^\nu f(y',t\rho)d\rho.$$
Next, we fix a $\nu\le m$. By 
$$|\partial_t^\nu f(y',t)|\le [\partial_t^{m}  f]_{C^\alpha(\bar G_r)}t^{m-\nu+\alpha},$$ 
we have, 
for any $(y',t)\in G_r$,  
\begin{align*}
|\partial_t^\nu F(y',t)|&\le [\partial_t^{m}  f]_{C^\alpha(\bar G_r)}t^{m-\nu+\alpha}
\int_0^1\rho^{m+\alpha-a-1}d\rho\\
&=\frac{1}{m+\alpha-a}[\partial_t^{m}  f]_{C^\alpha(\bar G_r)}t^{m-\nu+\alpha}.\end{align*}

For the H\"older semi-norm of $\partial_t^\nu F$ in $t$, it suffices to consider the case 
$\nu\ge m$. For any $(y',t_1)$ and $(y',t_2)$, we have 
$$|\partial_t^\nu f(y',t_1\rho)-\partial_t^\nu f(y',t_2\rho)|
\le [\partial_t^\nu f]_{C^\alpha(\bar G_r)}\rho^\alpha|t_1-t_2|^{\alpha}.$$
Then, 
\begin{align*}
|\partial_t^\nu F(y',t_1)-\partial_t^\nu F(y',t_2)|&\le [\partial_t^\nu f]_{C^\alpha(\bar G_r)}
|t_1-t_2|^\alpha\int_0^1\rho^{\nu+\alpha-a-1}d\rho\\
&=
\frac{1}{\nu+\alpha-a}[\partial_t^\nu f]_{C^\alpha(\bar G_r)}|t_1-t_2|^\alpha.\end{align*}

Next, we consider the H\"older semi-norm of $\partial_t^\nu F$ in $y'$. 
Take any $(y'_1,t)$, $(y_2',t)$. Fix a $\nu\le m$ and note 
\begin{align*}
\partial_t^\nu f(y_1', t\rho)-\partial_t^\nu f(y_2',t\rho)=
\rho t\int_0^1\big(\partial_t^{\nu+1}f(y_1', st\rho)-\partial_t^{\nu+1} f(y_2',st\rho)\big)ds.
\end{align*}
By iterating this formula finitely many times, we have 
\begin{align*}
\big|\partial_t^\nu f(y_1', t\rho)-\partial_t^\nu f(y_2',t\rho)\big|\le 
(\rho t)^{m-\nu}\sup_{s\in [0,1]}\big|\partial_t^{m}f(y_1', st\rho)-\partial_t^{m} f(y_2',st\rho)\big|,
\end{align*}
and hence
\begin{align*}
&\left|\partial_t^{\nu}F(y'_1,t)-\partial_t^{\nu}F(y'_2,t)\right|\\
&\quad\leq t^{m-\nu}\int_0^1\rho^{m-a-1}
\sup_{s\in [0,1]}\big|\partial_t^{m}f(y_1', st\rho)-\partial_t^{m} f(y_2',st\rho)\big|d\rho.
\end{align*}
If $m>a$, then 
\begin{align*}\left|\partial_t^{\nu}F(y'_1,t)-\partial_t^{\nu}F(y'_2,t)\right|
&\leq [\partial_t^mf]_{C^\alpha(\bar G_r)}
|y_1'-y_2'|^{\alpha}\int_0^1\rho^{m-a-1} d\rho\\
&\le \frac{1}{m-a}[\partial_t^mf]_{C^\alpha(\bar G_r)}|y_1'-y_2'|^{\alpha}.\end{align*}
If $m=a$, we take an arbitrary $\delta\in (0,1)$. Then, 
$$\aligned 
&\big|\partial_t^{m}f(y_1', t\rho)-\partial_t^{m}f(y_2',t\rho) \big|\\
&\qquad\le \big|\partial_t^{m}f(y_1', t\rho)-\partial_t^{m}f(y_2',t\rho) \big|^{\delta}
\big|\partial_t^{m}f(y_1', t\rho)-\partial_t^{m}f(y_2',t\rho) \big|^{1-\delta}\\
&\qquad\le \big(|\partial_t^{m}f(y_1', t\rho)|+|\partial_t^{m}f(y_2',t\rho)| \big)^{\delta}
\big|\partial_t^{m}f(y_1', t\rho)-\partial_t^{m}f(y_2',t\rho) \big|^{1-\delta}\\
&\qquad\le 2[\partial_t^mf]_{C^\alpha(\bar G_r)}
t^{\delta\alpha}\rho^{\delta\alpha}|y_1'-y_2'|^{(1-\delta)\alpha},\endaligned$$
and hence 
\begin{align*}
\left|\partial_t^{\nu}F(y'_1,t)-\partial_t^{\nu}F(y'_2,t)\right|
&\leq 2[\partial_t^mf]_{C^\alpha(\bar G_r)}r^{\delta\alpha}
|y_1'-y_2'|^{(1-\delta)\alpha}\int_0^1\rho^{-1+\delta\alpha} d\rho\\
&\le \frac{2r^{\delta\alpha}}{\delta\alpha}[\partial_t^mf]_{C^\alpha(\bar G_r)}|y_1'-y_2'|^{(1-\delta)\alpha}.\end{align*}
Therefore, $\partial_t^{\nu}F\in C^{\epsilon}(\bar{G}_r)$ for any $\nu\le m$ 
and any $\epsilon\in (0,\alpha)$. 
Next, we consider $\nu\ge m+1$. We have
\begin{align*}
\big|\partial_t^\nu f(y_1', t\rho)-\partial_t^\nu f(y_2',t\rho)\big|\le 
[\partial_t^{\nu}f]_{C^\alpha(\bar G_r)}|y_1'-y_2'|^{\alpha},
\end{align*}
and then
\begin{align*}
\left|\partial_t^{\nu}F(y'_1,t)-\partial_t^{\nu}F(y'_2,t)\right|
\leq \frac{1}{\nu-a}[\partial_t^\nu f]_{C^\alpha(\bar G_r)}|y'_1-y'_2|^{\alpha}.\end{align*}
Hence, 
$\partial_t^\nu F\in C^\alpha(\bar G_1)$ for any $a+1\le \nu\le l$.
\end{proof}

\end{document}